\documentclass[11pt,reqno]{amsart}
\usepackage{amsmath}
\usepackage{amsthm}
\usepackage{graphicx}
\usepackage{amsfonts}
\usepackage{bm}
\usepackage{bbm}
\usepackage{textcomp}
\usepackage{hyperref}
\usepackage{amssymb}
\usepackage{mathrsfs}
\usepackage{enumerate}
\usepackage[margin=1.1in]{geometry}
\usepackage{xcolor}

\numberwithin{equation}{section}

\newtheorem{theorem}{Theorem}[section]
\newtheorem{lemma}[theorem]{Lemma}
\newtheorem{proposition}[theorem]{Proposition}
\newtheorem{corollary}[theorem]{Corollary}

\theoremstyle{definition}

\newtheorem{definition}[theorem]{Definition}
\newtheorem{remark}[theorem]{Remark}
\newtheorem{asmp}[theorem]{Assumption}

\def\d{{\mathrm d}}
\def\E{{\mathbb E}}

\def\R{{\mathbb R}}

\def\N{{\mathbb N}}
\def\P{{\mathbb P}}
\def\FF{{\mathbb F}}
\def\PP{{\mathcal P}}

\def\GG{{\mathbb G}}
\def\G{{\mathcal G}}
\def\J{{\mathcal J}}

\def\A{{\mathcal A}}

\def\S{{\mathcal S}}
\def\F{{\mathcal F}}

\def\H{{\mathcal{H}}}

\newcommand{\indep}{\perp \!\!\! \perp}

\begin{document}
\title{It\^o's formula for flows of conditional measures on semimartingales}

\author{Xin Guo}
\address{Department of Industrial Engineering and Operations Research, University of California, Berkeley, Berkeley, California, USA.}
\email{xinguo@berkeley.edu}

\author{Jiacheng Zhang}
\address{Department of Industrial Engineering and Operations Research, University of California, Berkeley, Berkeley, California, USA.}
\email{jiachengz@berkeley.edu}

\date{}
\maketitle

\begin{abstract}
Motivated by recent developments of mean-field systems with common noise, this paper establishes It\^o’s formula for flows of conditional probability measures under a common filtration associated with general semimartingales. It generalizes existing works on flows of conditional measures on It\^o processes and flows of deterministic measures on general semimartingales.  It is derived via the  cylindrical function approach. 
One key technical component involves constructing conditionally independent copies of stochastic processes, which allows for establishing the equivalence between stochastic integrals with respect to the conditional laws of semimartingales and the conditional expectation of stochastic integrals with respect to copies of semimartingales. 

As an application, a class of control problems for jump-diffusions with partial observations is analyzed.
  
\end{abstract}

\section{Introduction}

\subsection*{It\^o's formula} The classical Itô's formula for semimartingales, the stochastic equivalent of the chain rule in calculus, is a cornerstone of stochastic analysis \cite{ito1951formula}. It intrinsically links partial differential equations (PDEs) with diffusion processes. Recent advancements in mean-field game and control theories \cite{buckdahn2017mean,cardaliaguet2019master, carmona2018probabilistic} have led to its extension to flows of probability measures, initially for diffusions \cite{buckdahn2017mean} and later for discontinuous semimartingales
(\cite{guo2023ito,talbi2023dynamic}). This generalized Itô's formula is crucial for deriving master equations in mean-field games and Bellman dynamic programming equations for McKean–Vlasov control problems, see \cite{carmona2018probabilistic} and the references therein.

 Most recent studies in mean-field theory involve controls and games with common noise (see for instance \cite{Ah,ARY,BZ,BFY,CDL,carmona2014master,CFS,CF,DV,GLL, hammersley2021weak,KT,Ko,kurtz1999particle,La,LW,Va})  or stochastic jump processes  (\cite{burzoni2020viscosity,fu2017mean,hafayed2014mean,hafayed2018optimal,hu2017singular,li2018mean}). In either case, It\^o's formula has been developed respectively in (\cite{buckdahn2017mean,carmona2018probabilistic}) and   (\cite{guo2023ito,talbi2023dynamic}). 
 In order to  analyze rigorously  the McKean–Vlasov dynamics with {\it both semimartingales and idiosyncratic and common noises}, it is necessary to establish the It\^o's formula for flows of {\it conditional laws} with semimartingales, which does not seem to  exist, to the best of our knowledge.  This is the primary focus of this paper.


\subsection*{Our work}
This paper establishes  It\^o's formula (Theorem \ref{mainthm}) for flow of conditional laws $\mu_t=$Law$(X_t|\G_t)$ driven by (possible discontinuous) semimartingales $\{X_t\}_{t\in[0,T]}$, with a general form of common noise represented by a sub-filtration $\GG=(\G_t)_{t\in[0,T]}$ and $\G=\G_T$. 

One key step  is to analyze the  It\^o integral $\int_0^t\eta_s\d X_s^\G$ with  $\{X_t^\G:=\E[X_t|\G]\}_{t\in[0,T]}$ over some appropriate adapted process $\{\eta_t\}_{t\in[0,T]}$. This is crucial as the integral of the conditional expectation such as $\int_0^t\eta_s\d X^\G_s$ is not necessarily equivalent to  the conditional expectation of the  integral of the form $\E\big[\int_0^t\eta_s\d X_s|\G\big]$ 
(See discussions and examples in Section \ref{section:condcopy}). 
Our analysis is  built on two key technical components: one is to construct   conditionally independent copies of general stochastic processes (Theorem \ref{thm:condindepcopy}),  the other is to establish the equivalence between the conditional expectation of stochastic integrals with respect to semimartingales and stochastic integrals
with respect to the conditional law of semimartingales, as well as    the equivalence between the quadratic variation of the conditional expectation of a semimartingale and the conditional expectation of
the quadratic variation of its two conditionally independent copies (Theorem \ref{thm:condprocess}).

With these key technical components, It\^o's formula for flows of conditional laws on semimartingales is established 
(Theorem \ref{mainthm}):  the result is first established for cylindrical functions, defined as smooth mean-field functions with integrable representations (see Definition \ref{def:cylin}). Since these functions are dense in the desired function space, the general form is then obtained by a localization argument combined with suitable versions of the dominated convergence theorem.

{    This  It\^o's formula (Theorem \ref{mainthm}) is applied to analyze control problems with partially observed jump-diffusions. An appropriate form of the verification theorem \ref{thm:verification} is derived, which recovers the results in \cite{bandini2019randomized} as  special cases. }

\subsection*{Related works} 
{  
One of the key technical elements is the construction of conditionally independent copies of stochastic processes. This concept  has been  treated informally, for instance, when defining a semimartingale  by stochastic differential equations involving both idiosyncratic and common noise in \cite{carmona2018probabilistic}.

Our construction of conditionally independent copies of general stochastic processes is inspired by the proofs of Yamada-Watanabe theorem (see for instance \cite{karatzas-shreve,kurtz2007yamada}) and the theory of weak and strong solutions of stochastic differential equations (see for instance \cite{jacod1981weak,kurtz2014weak}). 
Most existing methods specify the common information $\G$ as a stochastic input, and make use of the existence of the regular conditional distribution to construct an extended probability space. In contrast, we  construct  conditionally independent copies directly, rather than utilizing the regular conditional distribution. {(See the discussion at the beginning of Section \ref{sec:proofcondcopy} for a detailed comparison)}.

Our construction approach  generalizes the construction of conditionally independent copies to a broad class of  semimartingales and allows for  an arbitrary  sub $\sigma$-algebra. 
This result  may be of independent theoretical interest, besides its essential role in our derivation of It\^o's formula for flow of conditional laws on semimartingales.

A crucial assumption in our analysis is a condition on the filtration known as the compatibility condition.
This condition appears in various contexts under different names, such as \textit{immersion} \cite{jeanblanc2010immersion}, \textit{very good extensions} \cite{jacod1981weak}, and \textit{natural extensions} \cite{nicole1987compactification}. In the theory of  filtration enlargement, it is known as the \textit{H-hypothesis} (\cite{bremaud1978changes,elliott2000models}),  which plays a central role for characterizing semimartingale properties with filtration expansion or shrinkage.  Following \cite{kurtz2014weak}, we adopt  the term compatibility to be aligned with recent literature on mean-field games (\cite{beiglbock2018denseness,carmona2018probabilistic,carmona2016mean,carmona2017mean,lacker2016general}). 
}


Finally, it is worth pointing out that there are several methods to derive It\^o’s formula and its variants, including the time discretization  approach in \cite{buckdahn2017mean} and  \cite{li2018mean} for mean-field jump diffusions, the density approach in  \cite{cardaliaguet2019master} using the Fokker–Planck equation, and the particle approximation approach of \cite{carmona2018probabilistic,chassagneux2022probabilistic}, \cite{dos2023ito,reis2021relation}  by flows of empirical measures.
The cylindrical function approach has proved appropriate and powerful  for analysis of general semimartingales. It is  initially explored for Fleming–Viot processes \cite{fleming1979some},  later adopted in the  analysis of 
polynomial diffusions  \cite{cuchiero2019probability}, and for It\^o's formulas on flow of measures with   (discontinuous) semimartingales \cite{guo2023ito} and for 
non-anticipative maps of c\`adl\`ag rough paths \cite{cuchiero2025functional}.

\subsection*{Notation} Throughout the paper, we will adopt the following notations, unless otherwise specified.
\begin{itemize}
    \item { We use the bold letter $\bm x,\bm X$ to denote a (random) vector with multiple dimensions or a matrix depending on the context. $|\bm X|$ denotes the Euclidean norm of a vector $\bm X\in\R^d$, or the Frobenius norm of a matrix.}
    \item $\PP(\R^d)$ denotes the space of probability measures on $\R^d$ with the topology of weak convergence. $\PP_p(\R^d)$ for $p\in[1,\infty)$ denotes the probability measures with finite $p$-th moment, equipped with the Wasserstein-$p$ metric.
    \item $C^k(\R^d)$ is the set of all $k$-th differentiable functions on $\R^d$ with continuous derivatives up to the $k$-th order, and  $C_b^k(\R^d)\subset C^k(\R^d)$ is the set of all $k$-th differentiable functions on $\R^d$ with bounded and continuous derivatives up to the $k$-th order, with the convention $C(\R^d)=C^0(\R^d)$ and $C_b(\R^d)=C^0_b(\R^d)$.
    \item For a $d$-dimensional vector-valued random variable $\bm X$, $\|\bm X\|_{L^p}$ represents its $p$-th moment, i.e., $\|\bm X\|_{L^p}=\E[|\bm X|^p]^{\frac1p}$.
    \item For vectors $\bm a=(a_i)_{i=1,..,d},\bm b=(b_i)_{i=1,..,d}\in\R^d$, denote $\bm a\cdot\bm b=\sum_{i=1}^da_ib_i$. For matrices $\bm C=(C_{ij})_{i,j=1,..,d},\bm D=(D_{ij})_{i,j=1,..,d}\in\R^{d\times d}$,  denote $\bm C:\bm D=\sum_{i,j=1}^dC_{ij}D_{ij}$.
    \item For $\mu \in \PP(\R^d)$ and $\varphi : \R^d\to \R$ such that $\int_{\R^d}|\varphi(\bm x)|\d\mu(\bm x)<\infty$,  set
    $$\langle \mu,\varphi\rangle:=\int_{\R^d}\varphi(\bm x)\d\mu(\bm x).$$
    \item { For a c\`adl\`ag process $\bm X_t$, we adopt the standard notation:
    $$
    \bm X_{t-}:=\lim_{s\nearrow t}\bm X_s,\quad \Delta \bm X_t:=\bm X_t-\bm X_{t-},
    $$
    and $\bm X^c$ to represent the continuous part of the process.}
\end{itemize}


\section{Assumptions and main result}
Fix $T>0$ and $ d\in\mathbb N^{+}$.  
Consider two $\R^d$-valued $\FF$-semimartingales $\bm X=(\bm X_t)_{t\in[0,T]}$ and $\bm Y=(\bm Y_t)_{t\in[0,T]}$ 
on a completed filtered probability space $(\Omega,\F,\FF=(\F_t)_{t\in[0,T]},\P)$, and $\mu=\big\{\mu_t:=\text{Law}(\bm X_t|\G_t)\big\}_{t\in[0,T]}$ for  a sub-filtration $\GG=(\G_t)_{t\in[0,T]}\subset \FF$.  
Our focus is to establish an It\^o's formula for the  functional $\Phi(\mu_t, \bm Y_t)$, with  $\Phi:\PP_p(\R^d)\times\R^d\to\R$ for some $p\in[2,\infty)$  under suitable regularity conditions.

\subsection{Assumptions and preliminary tools}
Throughout the paper, we will make the following assumptions on the semimartingales $\bm X$ and $\bm Y$, the filtrations $\GG$ and $\FF$, and the function $\Phi$. These are standard assumptions to ensure that stochastic integrals are properly defined (\cite{protter2005stochastic}) with the presence of common noise (\cite{lacker2022superposition}), and for flow of measures on semimartingales (\cite{cox2024controlled,guo2023ito}).

First, we specify the assumptions on the semimartingales $\bm X$ and $\bm Y$.
\begin{asmp}\label{asp:semi}
For some $1\leq p\leq \infty$,
\begin{itemize}
    \item    
\begin{equation}\label{assumpt_XY}
    \begin{aligned}
    \big\|\bm X\big\|_{\H^p}:=\inf_{\bm M,\bm V}\bigg\||\bm X_0|+\sqrt{[\bm M,\bm M]_T}+\int_0^T\big|\d \bm V_s\big|\bigg\|_{L^p}<\infty,
    \\
    \big\|\bm Y\big\|_{\H^p}:=\inf_{\bm N,\bm U}\bigg\||\bm Y_0|+\sqrt{[\bm N,\bm N]_T}+\int_0^T\big|\d \bm U_s\big|\bigg\|_{L^p}<\infty,
    \end{aligned}
\end{equation}
where the infimum is taken over all possible decompositions  
\begin{equation*}
    \bm X_t = \bm X_0 + \bm M_t + \bm V_t\text{ and }\bm Y_t = \bm Y_0 + \bm N_t + \bm U_t,
\end{equation*} 
where $\bm V=(\bm V_t)_{t\in[0,T]}$ and $\bm U=(\bm U_t)_{t\in[0,T]}$ are adapted c\`adl\`ag processes of  finite variation with $\bm V_0 = \bm U_0 = \bm 0$, and $\bm M=(\bm M_t)_{t\in[0,T]}$ and  $\bm N=(\bm N_t)_{t\in[0,T]}$ are $\FF$-local martingales  such that $\bm M_0 = \bm N_0 = \bm 0$.
\item 
\begin{equation}\label{assumpt_XY2}
    \E\bigg[\Big(\sum_{0\leq s\leq T}|\Delta \bm X_s|\Big)^p\bigg]+\E\bigg[\Big(\sum_{0\leq s\leq T}|\Delta \bm Y_s|\Big)^p\bigg]<\infty.
\end{equation}
\end{itemize}
\end{asmp}


There are two propositions that are necessary for the subsequent analysis, especially 
for the localization argument in establishing the It\^o's lemma for the flow of conditional measures on semimartingales, as will be clear in Section \ref{subsec:local}. 

\begin{proposition}[{\cite[Section V, Theorem 2]{protter2005stochastic}}]\label{prop:sleqh}
For any $1\leq p\leq\infty$, define a norm on the space $D$ of all $\R$-valued adapted c\`adl\`ag processes, such that for $H=(H_t)_{t\in[0,T]}\in D$,  

\begin{equation*}
    \|H\|_{\S^p}:=\Big\|\sup_{0\leq t\leq T}|H_t|\Big\|_{L^p}.
\end{equation*}
Then, there exists a constant $c_p$ depending only on $p$ such that
\begin{equation*}
    \|H\|_{\S^p}\leq c_p\|H\|_{\H^p},
\end{equation*}
for any $\R$-valued semimartingale $H$.
\end{proposition}

\begin{proposition}[Emery inequality, {\cite[Section V, Theorem 3]{protter2005stochastic}}]\label{prop:emery}
Let $Z$ be an $\R$-valued  semimartingale, $H$ be an $\R$-valued adapted c\`adl\`ag process,
and $\frac1p+\frac1q=\frac1r$ ($1\leq p,q\leq \infty)$, then
\begin{equation*}
\Big\|\int_0^\cdot H_{s-}\d Z_s\Big\|_{\H^r}\leq\|H\|_{\S^p}\|Z\|_{\H^q}. 
\end{equation*}
\end{proposition}
\begin{remark}
Propositions \ref{prop:sleqh} and \ref{prop:emery} also hold for $\R^d$-valued semimartingales with  appropriate forms in $d\in\N^+$.
\end{remark}

Moreover, a proper form of dominated convergence theorem has been established as follows.

\begin{lemma}\label{lem:CDCT}{(\cite[Page 273, Lemma]{protter2005stochastic})}
    Let $p, q, r$ be given such that $\frac1p+\frac1q=\frac1r$ where $1 < r < +\infty$. Let $Z$ be a semimartingale such that $\|Z\|_{\mathcal{H}_q}<+\infty$, and let $\{H^n=(H^n_t)_{t\in[0,T]}\}_{n\in\N}$ and $Y=(Y_t)_{t\in[0,T]}$ be semimartingales such that $\|H^n\|_{\S^p}<+\infty$ for all $n\in\N$, $\|Y\|_{\S^p}<+\infty$ and $|H^n_t|\leq Y_t$. Suppose $\lim_{n\to\infty} H^n_{t-}(\omega) = 0$, for all $(t,\omega)$. Then 
    $$\lim_{n\to\infty} \bigg\|\int_0^TH^n_{s-}\d Z_s\bigg\|_ {\H^r}= 0.$$
\end{lemma}

Next, we impose the compatibility condition for  filtrations, in the following sense.


\begin{asmp}[Compatibility]\label{asp:filt}
    The sub-filtration $\GG\subset\FF$ satisfies the compatibility condition, if for each $0\leq t\leq T$,  $\F_t$ and $\G_T$ are conditionally independent given $\G_t$, expressed as 
$\F_t\indep\G_T\big|\G_t.$
\end{asmp}
{ 
\begin{remark}
\label{rmk:hypothesis}

 There are a number of alternative characterizations for the compatibility condition, such as the  \textit{immersion} or the \textit{H-hypothesis}, in the literature of filtration enlargement and shrinkage (see for instance \cite{bremaud1978changes,elliott2000models,jeanblanc2010immersion}). Given  two filtrations $\GG = (\G_t)_{t\in[0,T]}$ and $\FF=(\F_t)_{t\in[0,T]}$ defined on the same space, with $\GG\subset\FF$, the following are  known to be equivalent (\cite[Theorem 3]{bremaud1978changes}):
\begin{itemize}
    \item $\F_t\indep\G_T|\G_t$, for all $0\leq t\leq T.$ (Compatibility condition).
    \item Every square integrable $\GG$-martingale is  an $\FF$-martingale. (H-hypothesis).
    \item $\E[Z|\G_t ]$ = $\E[Z |\F_t ]$ a.s., for each $t \in [0,T]$ and each bounded $\F_T$-measurable random variable $Z$.
\end{itemize}
In the context of mean-field games,  this compatibility assumption ensures a weak closure of adapted processes, as  first noted in \cite[Lemma 3.11]{carmona2016mean} and  stated more explicitly in \cite[Theorem 5.4]{beiglbock2018denseness}. 
\end{remark}}


Moreover, as shown in the following lemma, this compatibility condition allows us to define a c\`adl\`ag version of $\mu_t$ by taking $\mu_t=\text{Law}(\bm X_t|\G_T)$, which will be used throughout the paper.
\begin{lemma}
    Under Assumption \ref{asp:filt}, $\text{Law}(\bm X_t|\G_T)=\text{Law}(\bm X_t|\G_t)$ almost surely.
\end{lemma}
\begin{proof}
It suffices to show that 
 for any bounded $\G_T$- measurable random variable $A$, 
    $
    \E\big[\varphi(\bm X_t)A\big]=\E\Big[\E\big[\varphi(\bm X_t)|\G_t\big]A\Big].$
    This is guaranteed by  $\F_t\indep\G_T|\G_t$ and      \begin{align*}\E\big[\varphi(X_t)A\big]&=\E\Big[\E\big[\varphi(X_t)A|\G_t\big]\Big]=\E\Big[\E\big[\varphi(X_t)|\G_t\big]\E\big[A|\G_t\big]\Big] 
\\&=\E\Big[\E\Big[\E\big[\varphi(X_t)|\G_t\big]A\big|\G_t\Big]\Big]=\E\Big[\E\big[\varphi(X_t)|\G_t\big]A\Big].\end{align*}
\end{proof}

 Finally, we will assume $\Phi \in 
C^{2,2}(\PP_p(\R^d)\times\R^d)$ 
according to the following definition.

\begin{definition}[$C^{2,2}(\PP_p(\R^d)\times\R^d)$]
For $p\geq 2$, a function $\Phi(\mu,\bm y)$ is  $C^{2,2}(\PP_p(\R^d)\times\R^d)$ if for any $\mu\in\mathcal{P}_p(\R^d)$ and $\bm y\in\R^d$, there exists a continuous mapping $(\mu,\bm y,\bm x_1)\to\frac{\delta \Phi}{\delta \mu}(\mu,\bm y,\bm x_1)$, as well as a continuous mapping $(\mu,\bm y,\bm x_1,\bm x_2)\to\frac{\delta^2 \Phi}{(\delta \mu)^2}(\mu,\bm y,\bm x_1,\bm x_2)$ that is symmetric in its two arguments $(\bm x_1,\bm x_2)$ and with the following properties:
\begin{itemize}
\item {(Continuous differentiability)} $\nabla_{\bm y} \Phi(\mu,\bm y)$, $\nabla_{\bm y}^2 \Phi(\mu,\bm y)$, $\frac{\delta \Phi}{\delta \mu}(\mu,\bm y,\bm x_1)$, $\nabla_{\bm x_1}\frac{\delta \Phi}{\delta\mu}(\mu,\bm y,\bm x_1)$, \\
$\nabla_{\bm y}\frac{\delta \Phi}{\delta\mu}(\mu,\bm y,\bm x_1)$ $\nabla^2_{\bm x_1}\frac{\delta \Phi}{\delta\mu}(\mu,\bm y,\bm x_1)$, $\nabla_{\bm x_1}\nabla_{\bm y}\frac{\delta \Phi}{\delta\mu}(\mu,\bm y,\bm x_1)$, $\nabla_{\bm x_1}\nabla_{\bm x_2}\frac{\delta^2 \Phi}{(\delta \mu)^2}(\mu,\bm y,\bm x_1,\bm x_2)$ all exist and are continuous  for all $\bm y,\bm x_1,\bm x_2\in\R^d$, $\mu\in \PP_p(\R^d)$.
\item {(Uniform polynomial-growth)} there exists a constant $c>0$ such that for all $\bm x_1,\bm x_2,\bm y\in \R^d$, $\mu\in \PP(\R^d)$, 
\begin{equation*}
    \begin{aligned}
     \Big|\nabla_{\bm y}\Phi\Big|\leq c\big(1+|\bm y|^{p-1}\big),\quad\Big|\nabla^2_{\bm y}\Phi\Big|&\leq c\big(1+|\bm y|^{p-2}\big),
    \\
    \bigg|\nabla_{\bm x_1}\frac{\delta \Phi}{\delta\mu}\bigg|+\bigg|\nabla_{\bm y}\frac{\delta \Phi}{\delta\mu}\bigg|&\leq c\big(1+|\bm x_1|^{p-1}+|\bm y|^{p-1}\big),
    \\
    \bigg|\nabla^2_{\bm x_1}\frac{\delta \Phi}{\delta\mu}\bigg|+\bigg|\nabla_{\bm x_1\bm y}\frac{\delta \Phi}{\delta\mu}\bigg|&\leq c\big(1+|\bm x_1|^{p-2}+|\bm y|^{{p-2}}\big),
    \\
    \bigg|\nabla_{\bm x_1,\bm x_2}\frac{\delta^2 \Phi}{(\delta \mu)^2}\bigg|&\leq c\big(1+|\bm x_1|^{p-2}+|\bm x_2|^{p-2}+|\bm y|^{p-2}\big).
    \end{aligned}
\end{equation*}
\item {(Fundamental theorem of calculus)} for everything $\mu,\nu\in\PP_2(\R^d)$, 
\begin{equation*}
    \begin{aligned}
        \Phi(\mu,\bm y)-&\Phi(\nu,\bm y)=\int_0^1\int_{\R^d}\frac{\delta \Phi}{\delta \mu}\big(\lambda \mu+(1-\lambda)\nu,\bm y, \bm x_1\big)(\mu-\nu)(\d \bm x_1)\d \lambda,
        \\
        \Phi(\mu,\bm y)-&\Phi(\nu,\bm y)-\int_0^1\int_{\R^d}\frac{\delta \Phi}{\delta \mu}\big(\nu,\bm y,\bm x_1\big)(\mu-\nu)(\d \bm x_1)
        \\
        &=\int_0^1\int_0^t\int_{\R^d}\frac{\delta^2\Phi}{(\delta \mu)^2}\big(s \mu+(1-s)\nu,\bm y, \bm x_1,\bm x_2\big)(\mu-\nu)(\d \bm x_1)(\mu-\nu)(\d \bm x_2)\d s\d t.
    \end{aligned} 
\end{equation*}
\end{itemize} 
Here  $\frac{\delta \Phi}{\delta \mu}$ and $\frac{\delta^2 \Phi}{(\delta \mu)^2}$ are respectively  (a version of) the \textit{first order} and the \textit{second order} linear derivatives of $\Phi$.
\end{definition}

{   {\begin{remark}
Note that the linear derivative is in the sense of \cite{cox2024controlled}, which is properly defined up to an additive constant.
This form of linear derivative is shown to be appropriate for analyzing the flow of measures on semimartingales \cite{cox2024controlled,guo2023ito}. 
    
\end{remark}}
}

\subsection{Main result}
Now we are ready to state the main result.

\begin{theorem}[It\^o's formula for flows of conditional measure on semimartingales]\label{mainthm}
   Given a completed filtered probability space 
 $(\Omega,\F,\FF=(\F_t)_{t\in[0,T]},\P)$, which supports two $\R^d$-valued $\FF$-semimartingales $\bm X=(\bm X_t)_{t\in[0,T]}$ and $\bm Y=(\bm Y_t)_{t\in[0,T]}$ satisfying Assumption \ref{asp:semi} for some $p\geq 2$. Let $\mu_t=$Law$(X_t|\G_t)$, with  the subfiltration $\GG=(\G_t)_{t\in[0,T]}\subset \FF$ satisfying Assumption \ref{asp:filt}. Then, for $\Phi\in C^{2,2}(\mathcal{P}_p(\R^d)\times\R^d)$, we have
   \begin{equation}\label{mainito}
   \begin{aligned}
   &\Phi(\mu_t,\bm Y_t)-\Phi(\mu_0,\bm Y_0)
   \\
   =&\;\overline\E\bigg[\int_{0+}^t\bigg(\nabla_{\bm x_1}\frac{\delta \Phi}{\delta\mu}\big(\mu_{s-},\bm Y_{s-},\bm X'_{s-}\big)\cdot\d (\bm X')^c_s+\frac12\nabla^2_{\bm x_1\bm x_1}\frac{\delta \Phi}{\delta\mu}\big(\mu_{s-},\bm Y_{s-},\bm X'_{s-}\big):\d [\bm X',\bm X']^c_s
   \\
   &\qquad\quad+\frac12\nabla^2_{\bm x_1\bm x_2}\frac{\delta^2 \Phi}{(\delta\mu)^2}\big(\mu_{s-},\bm Y_{s-},\bm X'_{s-},\bm X''_{s-}\big):\d [\bm X',\bm X'']^c_s+\nabla^2_{\bm x_1\bm y}\frac{\delta \Phi}{\delta\mu}\big(\mu_{s-},\bm Y_{s-},\bm X'_{s-}\big):\d [\bm X',\bm Y]^c_s\bigg)
   \\
   &\quad\quad+\sum_{0<s\leq t}\bigg(\frac12\bigg(\frac{\delta^2 \Phi}{(\delta\mu)^2}\big(\mu_{s-},\bm Y_{s-},\bm X'_{s},\bm X''_{s}\big)-\frac{\delta^2 \Phi}{(\delta\mu)^2}\big(\mu_{s-},\bm Y_{s-},\bm X'_{s-},\bm X''_{s}\big)
    \\
    &\qquad\qquad\qquad-\frac{\delta^2 \Phi}{(\delta\mu)^2}\big(\mu_{s-},\bm Y_{s-},\bm X'_{s},\bm X''_{s-}\big)+\frac{\delta^2 \Phi}{(\delta\mu)^2}\big(\mu_{s-},\bm Y_{s-},\bm X'_{s-},\bm X''_{s-}\big)\bigg)
    \\
    &\qquad\qquad\qquad+\bigg(\nabla_{\bm y}\frac{\delta \Phi}{\delta\mu}\big(\mu_{s-},\bm Y_{s-},\bm X'_{s}\big)-\nabla_{\bm y}\frac{\delta \Phi}{\delta\mu}\big(\mu_{s-},\bm Y_{s-},\bm X'_{s-}\big)\bigg)\cdot\Delta \bm Y_s
    \\
    &\quad\qquad\qquad+\frac{\delta \Phi}{\delta\mu}\big(\mu_{s-},\bm Y_{s-},\bm X'_s\big)-\frac{\delta \Phi}{\delta\mu}\big(\mu_{s-},\bm Y_{s-},\bm X'_{s-}\big)\bigg)\mathbbm {1}_{\{\mu_s=\mu_{s-}\}}
   \bigg|\F\bigg]
   \\
   &+\int_{0+}^t\nabla_{\bm y}\Phi(\mu_{s-},\bm Y_{s-})\cdot \d (\bm Y)^c_s+\frac12\int_{0+}^t\nabla_{\bm y\bm y}^2 \Phi(\mu_{s-},\bm Y_{s-}):\d [\bm Y,\bm Y]^c_s
   \\
   &+\sum_{0<s\leq t}\Big(\Phi(\mu_s,\bm Y_s)-\Phi(\mu_{s-},\bm Y_{s-})\Big),
   \end{aligned}
   \end{equation}
   for all $t\in[0,T]$, where $\bm X',\bm X''$ are two conditionally independent copies of $\bm X$ given the sub $\sigma$-algebra $\G_T$ defined in some enlarged probability space $(\overline\Omega,\overline{\F},\overline{\P})$, and $\overline{\E}[\cdot|\F]$ is the conditional expectation given $\F$ in the enlarged probability space.
\end{theorem}


{  
\begin{remark}

To ensure that the enlarged space $(\overline\Omega,\overline{\F},\overline{\FF},\overline{\P})$ is properly defined such that $\mu_t,\bm X_t,\bm Y_t$ and $\F$ are naturally extended to this enlarged space and $\bm X'$,$\bm X''$ are well-defined, one needs the precise notion of conditionally independent copies of stochastic processes.
In addition, one needs to show the existence of these conditionally  independent copies
in Theorem \ref{mainthm}.
 These will be studied  in the next section and specified in Theorem \ref{thm:condindepcopy} and Corollary \ref{cor:condindepcopy}.

\end{remark}

\begin{remark}
    Theorem \ref{mainthm} includes the following It\^o's formulas as special cases:
\begin{itemize}
    \item The It\^o's formula for mean-field games and control problems with common noise  developed in \cite{buckdahn2017mean,carmona2018probabilistic}: take $X_t$ as an It\^o's process driven by both idiosyncratic noise $W$ and common noise $W^0$, let the common filtration be the filtration generated by $W_0$.
    \item The It\^o's formula for flows of measures associated with possibly discontinuous processes without common noise studied in \cite{guo2023ito,talbi2023dynamic}:  set $X_t$ as a general semimartingale and choose $\G$ as the trivial $\sigma$-algebra.
    \item The It\^o's formula for partially observed diffusions as in \cite{bandini2019randomized}:  let $X_t$ be an It\^o's process and letting $\G$ be the filtration generated by the observation process. For more details, see  Section~\ref{sec:appl} for an extension of this problem to jump-diffusion processes.
\end{itemize}
\end{remark}
}


\section{Conditionally Independent Copy of Stochastic Process}\label{section:condcopy}

\subsection{Why conditionally independent copy?}


One of the key components in  Theorem \ref{mainthm} is  the notion  of   conditionally independent copies of stochastic processes.  
We now illustrate this concept through several examples   in the context of flow of conditional laws for semimartingales and its role in the derivation of It\^o's formula. 



Let us first consider a real-valued stochastic process  $X_t$ of the form
\begin{equation*}\label{counterexample}
    \d X_t=b_t\d t+\sigma_t\d W_t+\sigma^0_t\d W^0_t,
\end{equation*}
where $W,W^0$ are two independent real valued Brownian motions, $\GG=\{\G_t\}_{t\in[0,T]}$ is the  filtration generated by $W^0$, $\G=\G_T=\sigma(W^0_u, u\le T)$, and $\sigma_t$ and  $\sigma_t^0$ are appropriate functions to ensure the well-definedness of the above SDE. 
By the classical It\^o's formula, the dynamic of $g(X_t)$ for a smooth function $g$ is given by
\begin{equation*}
    \d g(X_t)=b_tg'(X_t)+\frac12(\sigma_t^2+(\sigma^0_t)^2)g''(X_t)\d  t+\sigma_tg'(X_t)\d W_t+\sigma^0_tg'(X_t)\d W^0_t,
\end{equation*}
and consequently the dynamic of $Z_t=\E[g(X_t)|\G]$ is given by
\begin{equation}\label{equ:dyy}
    \d Z_t=\E\bigg[b_tg'(X_t)+\frac12(\sigma_t^2+(\sigma^0_t)^2)g''(X_t)\Big|\G\bigg]\d  t+\E\big[\sigma^0_tg'(X_t)|\G\big]\d W^0_t.
\end{equation}
The above equation follows essentially from the fact that for progressively measurable process $(\eta_t)_{t\in [0,T]},$
\begin{equation}\label{equ:simp}
    \E\bigg[\int_0^t\eta_s\d W_s\Big|\G\bigg]=0,\qquad\E\bigg[\int_0^t\eta_s\d W^0_s\Big|\G\bigg]=\int_0^t\E[\eta_s|\G]\d W^0_s,\text{ for all }t\in[0,T]. 
\end{equation}
(See \cite[Lemma B.1]{lacker2022superposition} for more details). 

From \eqref{equ:dyy} one can calculate the quadratic variation of $Z_t$:
\begin{equation*}
    \d \langle Z\rangle_t=\E\big[\sigma^0_tg'(X_t)|\G\big]^2\d t=\E[\sigma^0_tg'(X_t)\widetilde\sigma^0_tg'(\widetilde X_t)|\G]\d t,
\end{equation*}
where $\widetilde\sigma_t^0,\widetilde X_t$ are respectively conditionally independent copies of $\sigma^0_t$ and $X_t$ given $\G$, as formally introduced in \cite{carmona2018probabilistic}. 
Moreover, one can derive for a smooth function $f$ the It\^o's formula for $f(Z_t)$, which depends on the precise expression of 
\begin{equation*}
    \begin{aligned}
    \int_0^tf'(Z_s)\d Z_s =&\int_0^tf'(Z_s) \bigg(\E\Big[b_sg'(X_s)+\frac12(\sigma_s^2+(\sigma^0_s)^2)g''(X_s)\Big|\G\Big]\d  s+\E\big[\sigma^0_tg'(X_s)|\G\big]\d W^0_s\bigg)
    \\
    =\int_0^t\bigg(\E&\Big[f'(Z_s) \Big(b_sg'(X_s)+\frac12(\sigma_s^2+(\sigma^0_s)^2)g''(X_s)\Big)\Big|\G\Big]\d  s+\E\big[f'(Z_s) \sigma^0_tg'(X_s)|\G\big]\d W^0_s\bigg)
    \end{aligned}
\end{equation*}
and 
\begin{equation*}
    \begin{aligned}
    \int_0^tf''(Z_s)\d \langle Z\rangle_s&= \int_0^tf''(Z_s)\cdot\frac12\E\big[\sigma^0_sg'(X_s)\widetilde\sigma^0_sg'(\widetilde X_s)\big|\G_T\big]\d s
    \\
    &=\frac12\int_0^t\E\big[f''(Z_s)\sigma^0_sg'(X_s)\widetilde\sigma^0_sg'(\widetilde X_s)\big|\G_T\big]\d s.
    \end{aligned}
\end{equation*}
However, there is an  issue when one considers  a general continuous semimartingale $X$ and a general sub $\sigma$-algebra $\G$, where the Itô's formula takes the form:
\begin{equation}\label{equ:generx}
    \d g(X_t)=g'(X_t)\d X_t+\frac12 g''(X_t)\d \langle X\rangle _t.
\end{equation}
Indeed, in order  to characterize the dynamics of $Z_t=\E[g(X_t)|\G]$, one needs an appropriate  expression for $\d Z_t$, $\d\langle Z\rangle_t$,
and more importantly for
\begin{equation}\label{equ:fy}
 \int_0^tf'(Z_s)\d Z_s \text{ and }\int_0^tf''(Z_s)\d \langle Z\rangle_s.
\end{equation}
One naive guess of the extension of  \eqref{equ:simp} to this general case is 
\begin{equation}\label{equ:wrong}
     \E\bigg[\int_0^t\eta_s\d X_s\Big|\G\bigg]=\int_0^t\E[\eta_s|\G] \d \E[X_s|\G].
\end{equation}
Unfortunately this guess does not hold in general. For instance, take  $X_t$ from \eqref{counterexample} with $b_t=0$. By \eqref{equ:simp}, the left hand side of \eqref{equ:wrong} becomes
\begin{equation*}
    \E\bigg[\int_0^t\eta_s\d X_s\Big|\G\bigg]=\E\bigg[\int_0^t\eta_s\big(\sigma_s\d W_s+\sigma^0_s\d W^0_s\big)\Big|\G\bigg]=\int_0^t\E\big[\eta_s\sigma^0_s\big|\G\big]\d W^0_s,
\end{equation*}
while the right hand side is
\begin{equation*}
    \int_0^t\E[\eta_s|\G] \d X^\G_s=\int_0^t\E[\eta_s|\G]\E[\sigma^0_s|\G]\d W^0_s,
\end{equation*}
and \eqref{equ:wrong} holds if and only if $\eta_s$ and $\sigma_s^0$ are conditionally uncorrelated, given $\G$.

This is precisely why one needs to introduce the (more general) concept of conditionally independent copies of stochastic processes. 
Intuitively,  if $X^{(1)}$, $X^{(2)}$ were two conditionally independent copies of $X$ given $\G$ satisfying $X^{(1)}\indep X^{(2)}|\G$ and Law$(X^{(1)}|\G)=$ Law$(X^{(2)}|\G)=$ Law$(X|\G)$ defining on some probability space that depends on $\G$, then we can get through  (\ref{equ:wrong}) such that 
\begin{equation*}
    \begin{aligned}
    \int_0^tf'(Z_s)\d Z_s&=\int_0^tf'(Z_s)\d \E\big[g(X_s)|\G\big]=\int_0^tf'(Z_s)\d \widetilde\E\big[g(X^{(1)}_s)\big]=\widetilde\E\bigg[\int_0^tf'(Z_s)\d g(X^{(1)}_s)\bigg]
    \\
    &=\widetilde\E\bigg[\int_0^tf'(Z_s)\Big(g'(X^{(1)}_s)\d X^{(1)}_s+\frac12g''(X^{(1)}_s)\d\langle X^{(1)}\rangle_s\Big)\bigg],
    \end{aligned}
\end{equation*}
and 
\begin{equation*}
    \begin{aligned}
    \int_0^tf''(Z_s)\d \langle Z\rangle_s&=\int_0^tf'(Z_s)\d \langle\E\big[g(X_\cdot)|\G\big]\rangle_s=\widetilde\E\bigg[\int_0^tf'(Z_s)\d \langle g(X^{(1)}_\cdot),g(X^{(2)}_\cdot)\rangle_s\bigg]
    \\
    &=\widetilde\E\bigg[\int_0^tf''(Z_s)g'(X^{(1)}_s)g'(X^{(2)}_s)\d\langle X^{(1)},X^{(2)}\rangle_s\bigg],
    \end{aligned}
\end{equation*} 
where $\widetilde\E$ denotes the expectation on $X^{(1)}$ and $X^{(2)}$ while keeping the information that is $\G$-measurable. With these expressions, one can analyze  \eqref{equ:fy} for a general semimartingale $X$, as will be  detailed  in \eqref{cpresult3} and \eqref{cpresult4} of Theorem \ref{thm:condprocess}.

Next, we will provide the construction of conditionally independent copies for the general case of a  semimartingale $\bm X$ and an {\it arbitrary}  sub $\sigma$-algebra $\G$ which is not  necessarily associated with any random variable. As emphasized earlier, this construction is necessary for establishing It\^o's formula with the flow of conditional measures on semimartigales.

\subsection{Construction of conditionally independent copies of   stochastic process}

Given a probability space $(\Omega,\F,\P)$, consider the extended probability space $\overline\Omega=\Omega^{n+1}$  defined by
\begin{equation}\label{def:probofexd}
\overline\Omega=\Omega^{n+1}=\{(\omega_0,\omega_1,\cdots,\omega_n)|\omega_0,\omega_1,\cdots,\omega_n\in\Omega\},
\end{equation}
and the extended $\sigma$-algebra $\overline\F$ by
\begin{equation}\label{def:probsigmaalexd}
\overline\F=\sigma\{A_0\times A_1\times \cdots\times A_n:A_0,A_1,\cdots,A_n\in\F\}.
\end{equation}
To  ensure that these copies are conditionally independent, let us define the probability measure $\overline \P$ as follows: for $A_0,A_1,\cdots,A_n\in\F$,
\begin{equation}\label{def:probpexd}
    \overline \P\big(A_0\times A_1\times \cdots\times A_n\big):=\E\bigg[\mathbbm{1}_{A_0}\prod_{i=1}^n\P(A_i|\G)\bigg].
\end{equation}
This measure is well  defined since the set $\{A_0\times A_1\times \cdots\times A_n\}$ is a $\pi$-system. Finally, we define the completion of the above extended probability space and still use $(\overline\Omega,\overline \F,\overline \P)$ to denote the completed probability space. Further,  define $\widetilde\F,\widetilde\G$ for the extension of $\F$ and $\G$ into this enlarged probability space as
\begin{equation}\label{def:prodfgexd}
    \widetilde\F:=\text{the completion of }\{A\times\Omega^n:A\in\F\}, \quad \widetilde\G:=\text{the completion of }\{A\times\Omega^n:A\in\G\}.
\end{equation}

Then we have:

\begin{theorem}\label{thm:condindepcopy}
    Given a probability space $(\Omega,\F,\P)$  and  a sub $\sigma$-algebra $\G\subset\F$, there exists an enlarged completed probability space $(\overline\Omega,\overline\F,\overline\P)$ satisfying \eqref{def:probofexd}, \eqref{def:probsigmaalexd} and \eqref{def:probpexd}; and for any random variable $R$ { taking values in a standard Borel space $(E,\mathcal{E})$}, there exist an $\widetilde R$ and $n$ conditionally independent copies $\{R^{(i)}\}_{i=1,2,..n}$ given $\widetilde\G$ such that 
    \begin{equation}\label{def:rcopy}
        \widetilde R(\omega_0,\omega_1,\cdots,\omega_n)=R(\omega_0),\;R^{(i)}(\omega_0,\omega_1,\cdots,\omega_n)=R(\omega_i),
    \end{equation}
    with  $\widetilde R$ and $\{R^{(i)}\}_{i=1,2,..n}$ satisfying the following:
    \begin{itemize}
        \item Given $\widetilde\G$, $\widetilde R$ and  $\big\{R^{(i)}\big\}_{i=1,2,..n}$ are conditionally independent. 
        \item Given $\widetilde\F$, $\widetilde R$ and  $\big\{R^{(i)}\big\}_{i=1,2,..n}$ are conditionally independent. 
        \item Law$\big(R\big|\G\big)(\omega_0)=$Law$\big(\widetilde R\big|\widetilde\G\big)(\overline\omega)=$Law$\big(R^{(i)}\big|{\widetilde\G}\big)(\overline\omega)=$Law$\big(R^{(i)}\big|{\widetilde\F}\big)(\overline\omega)$, for all $i=1,2,..,n$, a.s.,
    \end{itemize}
    for $\overline\omega=(\omega_0,,\omega_1,\cdots,\omega_n)\in\overline\Omega$.
\end{theorem}
Note that when $\G$ is a trivial $\sigma$-algebra $\{\emptyset,\Omega\}$, then the above theorem reduces to the standard construction of independent copies of stochastic processes.


\subsection{Construction of independent copies of semimartingale}
In the context of semimartingales, we can similarly define a corresponding version of the conditionally independent copy as follows. Take the completed filtered probability space $(\Omega,\F,\FF=(\F_t)_{t\in[0,T]},\P)$ with a sub-filtration $\GG=(\G_t)_{t\in[0,T]}\subset \FF$ satisfying $\F_t\indep\G_T|\G_t$,  define a complete extended probability space $(\overline\Omega,\overline\F,\overline\P)$ satisfying \eqref{def:probofexd}, \eqref{def:probsigmaalexd} and \eqref{def:probpexd} by setting the sub $\sigma$-algebra $\G=\G_T$.
 Define $\overline\FF:=(\overline\F_t)_{t\in(0,T)}$ by
 \begin{equation}\label{def:probofilexd}
\overline\F_t:=\text{completion of }\sigma\{A_0\times A_1\times \cdots\times A_n:A_0,A_1,\cdots,A_n\in\F_t\},
\end{equation}
and  define $\widetilde\FF:=(\widetilde\F_t)_{t\in(0,T)}$ and $\widetilde\GG:=(\widetilde\G_t)_{t\in(0,T)}$ respectively for the extension of $\FF$ and $\GG$ in this enlarged probability space as
\begin{equation}\label{def:probofilfgexd}
    \widetilde\F_t:=\text{completion of }\{A\times\Omega^n:A\in\F_t\}, \quad \widetilde\G_t:=\text{completion of }\{A\times\Omega^n:A\in\G_t\}.
\end{equation}

With this construction, we have
\begin{corollary}\label{cor:condindepcopy}
Given two $\R^d$-valued $\FF$-semimartingales $\bm X=(\bm X_t)_{t\in[0,T]}$ and $\bm Y=(\bm Y_t)_{t\in[0,T]}$  on a filtered probability space $(\Omega,\F,\FF=(\F_t)_{t\in[0,T]},\P)$, and a sub $\sigma$-algebra $\G\subset \F$, there exists an enlarged probability space $(\overline\Omega,\overline\F,\overline\FF=(\overline\F_t)_{t\in[0,T]},\overline\P)$ satisfying \eqref{def:probofexd}, \eqref{def:probsigmaalexd}, \eqref{def:probpexd} and \eqref{def:probofilexd} and natural extension $\widetilde \F$ and $\widetilde \G$ satisfying \eqref{def:prodfgexd}. Moreover, one can define 
two conditionally independent copies $(\bm X',\bm Y')$, $(\bm X'',\bm Y'')$ of the random processes $(\bm X,\bm Y)$ given $\G$ and natural extension  $(\widetilde{\bm X},\widetilde{\bm Y})$  by
\begin{equation*}
    (\bm X',\bm Y')(\omega_0,\omega_1,\cdots,\omega_n)=(\bm X,\bm Y)(\omega_1),\;(\bm X'',\bm Y'')(\omega_0,\omega_1,\cdots,\omega_n)=(\bm X,\bm Y)(\omega_2),
\end{equation*}
 $$
(\widetilde{\bm X},\widetilde{\bm Y})(\omega_0,\omega_1,\cdots,\omega_n)=(\bm X,\bm Y)(\omega_0),
$$
with $(\widetilde{\bm X},\widetilde{\bm Y})$, $(\bm X',\bm Y')$, and $(\bm X'',\bm Y'')$ satisfying 
\begin{equation}\label{condcopy}
\begin{cases}
    \text{Law}\big(\widetilde{\bm X},\widetilde{\bm Y}\big|\widetilde\G\big)=\text{Law}\big(\bm X',\bm Y'\big|\widetilde\G\big)=\text{Law}\big(\bm X',\bm Y'\big|\widetilde\F\big)
    \\
    \qquad=\text{Law}\big(\bm X'',\bm Y''\big|\widetilde\G\big)=\text{Law}\big(\bm X'',\bm Y''\big|\widetilde\F\big)\text{ a.s.,}
    \\
    (\widetilde{\bm X},\widetilde{\bm Y}), (\bm X',\bm Y'), (\bm X'',\bm Y'')\text{ are conditionally independent, given }\G \ or \ \F.

\end{cases}
\end{equation}
Furthermore, if the sub $\sigma$-algebra $\G=\G_T$, and $\F_t\indep\G_T\big|\G_t,$ for all $0\leq t\leq T$,  then the sub-filtration $\widetilde\G\subset\widetilde\F$ satisfies $\widetilde\F_t\indep\widetilde\G_T\big|\widetilde\G_t,$ for all $0\leq t\leq T$.
\end{corollary}

Note that  Theorem \ref{thm:condindepcopy} allows to use the notation $(\bm X,\bm Y),\F$, and $\G$ to respectively represent the natural extensions $(\widetilde {\bm X},\widetilde{\bm Y}),\widetilde \F$, and $\widetilde \G$ without any ambiguity, as  stated in the above corollary. Furthermore,  Corollary \ref{cor:condindepcopy} enables simplifying the notation $\widetilde\FF$ and $\widetilde\GG$ respectively by $\FF$ and $\GG$ without any ambiguity. Moreover, if we define the conditional law  in the enlarged probability space as $\widetilde \mu_t:=\text{Law}\big(\widetilde{\bm X}\big|\widetilde\G_t\big)$, then by the compatibility condition and Theorem \ref{thm:condprocess},
    \begin{equation*}
        \widetilde\mu_t(\omega_0,\omega_1,\cdots,\omega_n)=\text{Law}(\widetilde{\bm X_t}\big|\widetilde\G_T)(\omega_0,\omega_1,\cdots,\omega_n)=\text{Law}(\bm X_t\big|\widetilde\G_T)(\omega_0)=\mu_t(\omega_0),
    \end{equation*} 
which in turn allows for  using  $\mu_t$ for $\widetilde \mu_t$ without any ambiguity. Therefore, from now on we sometimes will simply use the original notation to denote the natural extensions.

\subsection{Examples}
To get some intuition for  this notion of  conditionally independent copies of stochastic processes,  let us see some special examples of $\F$ and the sub $\sigma$-algebra $\G$.
\vspace{5pt}
 
{ \noindent\textbf{Example 1.} Let $A,B$ be two independent random variables with $\F=\sigma(A,B)$ and $\G=\sigma(B)$. One can enlarge the probability space $(\overline\Omega,\overline \F,\overline\P)$ to include $\{(A^{(i)},B^{(i)})\}_{i=1,2,..,n}$ satisfying \eqref{def:probofexd}, \eqref{def:probsigmaalexd},\eqref{def:probpexd} and \eqref{def:rcopy}. Moreover, 
\begin{enumerate}
    \item $\{(A^{(i)},B^{(i)})\}_{i=1,2,..,n}$ have the same distribution as $(A,B)$.
    \item $B^{(i)}=B$ almost surely.
    \item $\overline \F=\sigma\big(A,B,\{A^{(i)},B^{(i)}\}_{i=1,2,..,n}\big)=\sigma\big(A,B,\{A^{(i)}\}_{i=1,2,..,n}\big)$.
    \item $\{A^{(i)}\}_{i=1,2,..,n}$ and $A,B$ are independent.
\end{enumerate}  
Consequently, for any $X=f(A,B)$, the conditionally independent copies of $X$ are $X^{(i)}=f(A^{(i)},B^{(i)})=f(A^{(i)},B), \ \ i=1, \cdots, n.$

\vspace{5pt}

\noindent\textbf{Example 2.} Let $X$ satisfy the following stochastic differential equation
\begin{equation*}   
    \d X_t=b(X_t)\d t+\sigma(X_t)\d W_t+\sigma^{(0)}(X_t)\d W^{(0)}_t,
\end{equation*}
where $W,W^{(0)}$ are two independent Brownian motions, and $b$, $\sigma$, and $\sigma^0$ satisfy appropriate continuity conditions. Let $\F_t=\sigma\big(\{W_s,W^{(0)}_s\}_{s\leq t}\big)$ and the sub-filtration $\G_t=\sigma\big(\{W^{(0)}_s\}_{s\leq t}\big)$, then one can enlarge the probability space to include independent Brownian motions $\big\{(X',W',{W^{(0)}}')\big\}_{i=1,2,..,N}$ satisfying \eqref{def:probofexd}, \eqref{def:probsigmaalexd},\eqref{def:probpexd}, \eqref{def:rcopy} and \eqref{def:probofilexd}. Moreover, 
\begin{enumerate}
    \item $X'$ satisfies the following stochastic differential equation
    \begin{equation*}   
        \d X'_t=b(X'_t)\d t+\sigma(X'_t)\d W'_t+\sigma^{(0)}(X'_t)\d (W^{(0)}_t)'.
    \end{equation*}
    \item $W^{(0)}=(W^{(0)})'$ almost surely.
    \item $\overline \F_t=\sigma\big(\{W_s,W^{(0)}_s,W'_s,(W^{(0}_s)'\}_{s\leq t}\big)=\sigma\big(\{W_s,W^{(0}_s,W'_s\}_{s\leq t}\big)$
    \item $W,W',W^{(0)}$ are independent Brownian motions.
\end{enumerate}
Consequently, $X'$ is one conditionally independent copy of $X$ given $\G$.

\vspace{5pt} }

There are more examples of conditionally independent copies of stochastic processes in the literature of mean-field games.  For instance,  when defining a semimartingale  by stochastic differential equations involving both idiosyncratic and common noise,  conditionally independent copies are  formally formulated in \cite{carmona2018probabilistic} as $\overline \Omega=\Omega_0\times\Omega_1\times\widetilde\Omega_1$, where the original probability space follows the structure $\Omega=\Omega_0\times\Omega_1$, and $\widetilde\Omega_1$ represents the independent copy of the probability space $\Omega_1$
 (\cite[Theorem 4.14, Volume II]{carmona2018probabilistic}).

\subsection{Properties of the conditionally independent copies of semimartingales}
 Next, we study  properties of these conditionally independent copies of semimartingales,  and explore the relation between the conditional expectation of  their stochastic integrals and   stochastic integrals  with respect to the conditional law of semimartingales. 
 For  notational simplicity, we will present the results for the one-dimensional case and their corresponding multi-dimensional cases can be obtained with appropriate adaptations.

 First, by the compatibility condition and the classical results regarding filtration enlargement, we have:
{ 

\begin{proposition}\label{prop:condhypo}
    Given an $\R$-valued $\FF$-semimartingales $\{X_t\}_{t\in[0,T]}$ on a filtered probability space $(\Omega,\F,\FF,\P)$. Let  $X'$ and  $X''$ be its conditionally independent copies defined in Proposition \ref{cor:condindepcopy}  in the enlarged filtered probability space $(\overline\Omega,\overline\F,\overline\FF=(\overline\F_t)_{t\in[0,T]},\overline\P)$ satisfying \eqref{def:probofexd}, \eqref{def:probpexd} and \eqref{def:probofilexd}. 
    Then, 
    \begin{enumerate}
        \item\label{cpresult0} $X$, $X'$, $X''$ are $\overline\FF$-semimartingales.
        \item Suppose $\E[|X_t|]<\infty$ for all $t\in[0,T]$, and define
        \begin{equation*}
            X^\G_t:=\overline\E\big[ X_t\big|\G_T\big]=\overline\E\big[X_t\big|\G_t\big].
        \end{equation*} 
        \begin{enumerate}
        \item\label{cpresult1} 
        If $\{X_t\}_{t\in[0,T]}$ is a finite variation process and 
        $\E\big[\big(\int_0^T|\d X_s|\big)^p\big]<\infty$
        for $p\geq1$, then $\{X^\G_t\}_{t\in[0,T]}$ is also a finite variation process, and
        \begin{equation*}
            \overline\E\bigg[\bigg(\int_0^T\big|\d X^\G_s\big|\bigg)^p\bigg]<\infty.
        \end{equation*}   
        \item\label{cpresult2} $\{X^\G_t\}_{t\in[0,T]}$ is a semimartingale with respect to $\overline\FF$, $\FF$, and $\GG$.
        \end{enumerate}
    \end{enumerate}
\end{proposition}

\begin{remark}
\label{rmk:H-hypothesis}
Proposition \ref{prop:condhypo} shows that semimartingales remain semimartingales under the enlarged filtration $\overline{\FF}$. This is closely related to the so called $H'$-hypothesis where any $\FF$-martingale is an $\overline\FF$-semimartingale with the enlargement filtration from 
$\FF$ to $\overline\FF$ (see for instance \cite{aksamit2017enlargement,grigorian2023enlargement}). Here we have obtained a slightly stronger result that any $\FF$-semimartingale $X$ and its conditionally  independent copy remain $\overline\FF$-semimartingales; meanwhile the $\GG$-semimartingale  $X^\G$ (the projection taken with respect to the $\sigma$-algebra $\G$ as in  \cite{bremaud1978changes}) remains $\overline\FF$-semimartingale. Despite the apparent similarity,  we are unable to locate exactly the same statement, and will provide the proof for completeness.  
\end{remark}

\begin{theorem}\label{thm:condprocess}
    Give two $\R$-valued $\FF$-semimartingales $\{X_t\}_{t\in[0,T]}$ and  $\{Y_t\}_{t\in[0,T]}$  on a filtered probability space $(\Omega,\F,\FF,\P)$. Let  $(X',Y')$ and  $(X'',Y'')$ be their conditionally independent copies defined in Proposition \ref{cor:condindepcopy}  in the enlarged filtered probability space $(\overline\Omega,\overline\F,\overline\FF=(\F_t)_{t\in[0,T]},\overline\P)$ satisfying \eqref{def:probofexd}, \eqref{def:probpexd} and \eqref{def:probofilexd}. 
    Suppose $\E[|X_t|]<\infty$ and $\E[|Y_t|]<\infty$ for all $t\in[0,T]$. 
    \begin{enumerate}
        \item\label{cpresult3} If $\{ Z_t\}_{t\in[0,T]}$ is an $\R$-valued $\FF$-adapted c\`adl\`ag process in the filtered probability space $(\Omega,\F,\FF,\P)$,  and $\|X\|_{\mathcal{H}_p}+\|Z\|_{\mathcal{H}_q}< \infty$ with $\frac1p+\frac1q=1$, then
        \begin{equation*}
            \int_0^t Z_{s-}\d  X^\G_s=\overline\E\bigg[\int_0^tZ_{s-}\d X'_s\bigg|\F\bigg].
        \end{equation*}
        \item\label{cpresult4} If $\{ Z_t\}_{t\in[0,T]}$ is an $\R$-valued $\FF$-adapted c\`adl\`ag  process in a filtered probability space $(\Omega,\F,\FF,\P)$ and  $\|X\|_{\mathcal{H}_p}+\|Y\|_{\mathcal{H}_q}+\|Z\|_{\mathcal{H}_r}< \infty$ with $\frac1p+\frac1q+\frac1r=1$, then
        \begin{equation*}
            \int_0^t Z_{s-}\d \big[X^\G,Y^\G\big]_s=\overline\E\bigg[\int_0^tZ_{s-}\d\big[X'',Y'\big]_s\bigg|\F\bigg],
        \end{equation*}
        \begin{equation*}
        \int_0^t Z_s\d\big[X^\G,Y\big]_s=\overline\E\bigg[\int_0^t Z_s\d\big[X',Y\big]_s\bigg|\F\bigg].
        \end{equation*}
    \end{enumerate}
\end{theorem}}



\begin{remark}\label{rmk:equivalence}
Statements \eqref{cpresult3} and \eqref{cpresult4} in Theorem \ref{thm:condprocess} are the key equations  for deriving  the It\^o's formula (Theorem \ref{mainthm}. They ensure crucially a) the equivalence between  the stochastic integral of the conditional expectation of a semimartingale and the conditional expectation of the It\^o's integral of its conditionally independent copy, and b)   the equivalence between  the quadratic variation of the conditional expectation of a semimartingale and the conditional expectation of the quadratic variation of its two conditionally independent copies. 
\end{remark}
The proof of Theorem \ref{thm:condprocess} is presented in Section \ref{sec:pfprocond}.

\section{Proof of Main Result (Theorem \ref{mainthm})}
Having established necessary properties of conditionally independent copies of semimartingales, we are now ready to prove Theorem \ref{mainthm}, the  It\^o's formula for flow of conditional laws on semimartingales, through several steps. First, we will derive the formula 
for cylindrical functions;
we will then show that cylindrical functions are dense in $C^{2,2}(\R^d\times\mathcal{P}_p(\R^d))$; finally,
with appropriate localization technique and  dominated convergence theorem, we will  finish the proof. 

\subsection{ Step 1: Ito's formula for cylindrical functions.}\label{sec:Itocyl}

First, let us recall the cylindrical functions (see for instance (\cite{cuchiero2019probability}).
\begin{definition}[$C^{2,2}$ cylindrical functions]\label{def:cylin}
A function $\Phi:\PP(\R^d)\times\R^d\to\R$ is a $C^{2,2}$ cylindrical function if 
\begin{equation*}
    \Phi(\mu,\bm y)=f\big(\langle \mu,g^{(1)}\rangle,..\langle \mu,g^{(n)}\rangle,\bm y\big),
\end{equation*}
where $f\in C^2(\R^{n+d})$ and $g=(g^{(1)},..,g^{(n)})$, with $g^{(i)}\in C^2_b(\R^d)$.
\end{definition}

We will first show that for smooth cylindrical functions $\Phi$, Theorem \ref{mainthm} holds.
Indeed, for semimartingales $\{\bm X_t\}_{t\in[0,T]}$ and $g^{(i)}\in C_b^2(\R^d)$ with $i=1,2,..,n$, by the classical It\^o's formula \cite[Theorem 33, Chapter II]{protter2005stochastic},
\begin{equation*}\label{equ:itog1}
    \begin{aligned}
    g^{(i)}(\bm X_t)-g^{(i)}(\bm X_0)
    =&\int_{0+}^t\nabla g^{(i)}(\bm X_{s-})\cdot\d \bm X^c_s+\frac12\int_{0+}^t\nabla^2 g^{(i)}(\bm X_{s-}):\d [\bm X,\bm X]^c_s
    \\
    &\;\;+\sum_{0<s\leq t}\Big\{g^{(i)}(\bm X_s)-g^{(i)}(\bm X_{s-})\Big\},\text{ for }t\in[0,T],
    \end{aligned}
\end{equation*}
and $\{g^{(i)}(\bm X_t)\}_{t\in[0,T]}$ for $i=1,..,n$ are semimartingales. Taking conditional expectation given $\G_T$, and recall that $\mu_t=\text{Law}(\bm X_t|\G_t)=\text{Law}(\bm X_t|\G_T)$, we have
\begin{equation}\label{equzk}
    \begin{aligned}
    \big\langle\mu_t,g^{(i)}\big\rangle-\big\langle\mu_0,g^{(i)}\big\rangle
    =&\;\E\bigg[\int_{0+}^t\nabla g^{(i)}(\bm X_{s-})\cdot\d \bm X^c_s+\frac12\int_{0+}^t\nabla^2 g^{(i)}(\bm X_{s-}):\d [\bm X,\bm X]^c_s
    \\
    &\qquad+\sum_{0<s\leq t}\Big\{g^{(i)}(\bm X_s)-g^{(i)}(\bm X_{s-})\Big\}\bigg|\G_T\bigg].
    \end{aligned}
\end{equation}
Define $\bm Z$ by 
\begin{equation*}
    \bm Z:=\Big\{\bm Z_t=\Big(\big\langle\mu_t,g^{(1)}\big\rangle,\big\langle\mu_t,g^{(2)}\big\rangle,..,\big\langle\mu_t,g^{(n)}\big\rangle\Big)\Big\}_{t\in[0,T]},
\end{equation*}
then by \eqref{cpresult2} of Theorem \ref{thm:condprocess},  $\bm Z$ is an $\overline\FF$-semimartingale. 

Let us for the moment assume that $\bm Y_t$ is bounded. Then applying It\^o's formula to  $f\big(\bm Z,\bm Y\big)$ yields
\begin{equation}\label{equitof}
    \begin{aligned}
    &f(\bm Z_t,\bm Y_t)-f(\bm Z_0,\bm Y_0)
    \\=&\underbrace{\int_{0+}^t\nabla_{\bm z} f(\bm Z_{s-},\bm Y_{s-})\cdot\d \bm Z_s}_{J_1}+\frac12\underbrace{\int_{0+}^t\nabla_{\bm z\bm z}^2 f(\bm Z_{s-},\bm Y_{s-}):\d [\bm Z,\bm Z]_s+\int_{0+}^t\nabla_{\bm z\bm y}^2 f(\bm Z_{s-},\bm Y_{s-}):\d [\bm Z,\bm Y]_s}_{J_2}
    \\
    \\
    -&\underbrace{\sum_{0<s\leq t}\Big\{\nabla_{\bm z} f(\bm Y_{s-},\bm Z_{s-})\cdot\Delta\bm Z_s+\frac12\nabla^2_{\bm z\bm z} f(\bm Y_{s-},\bm Z_{s-}):\big(\Delta\bm Z_s\Delta\bm Z_s^\top\big)+\nabla^2_{\bm z\bm y} f(\bm Y_{s-},\bm Z_{s-}):\big(\Delta\bm Z_s\Delta\bm Y_s^\top\big)\Big\}}_{J_3}
    \\
    &\quad+\int_{0+}^t\nabla_{\bm y} f(\bm Z_{s-},\bm Y_{s-})\cdot\d \bm Y^c_s+\frac12\int_{0+}^t\nabla_{\bm y\bm y}^2 f(\bm Z_{s-},\bm Y_{s-}):\d [\bm Y,\bm Y]^c_s
    \\
    &\quad+\sum_{0<s\leq t}\Big\{f(\bm Y_s,\bm Z_s)-f(\bm Y_{s-},\bm Z_{s-})\Big\}.
    \end{aligned}
\end{equation}
We will compute $J_1,J_2,$ and $J_3$ separately.  First,  note that both $g^{(i)}$ and $\bm Y_t$ are bounded, and $f\in C^{2}(\R^{n+d})$, therefore  $\nabla_{\bm z} f(\bm Z_{s-},\bm Y_{s-})$, $\nabla^2_{\bm z\bm y} f(\bm Z_{s-},\bm Y_{s-})$, and $\nabla^2_{\bm z} f(\bm Z_{s-},\bm Y_{s-})$  are bounded, and  the conditions for applying \eqref{cpresult3} and \eqref{cpresult4} of Theorem \ref{thm:condprocess} are satisfied. 

For the conditionally independent copy $\bm X'$, we have
\begin{equation*}\label{equ:itog}
    \begin{aligned}
    g^{(i)}(\bm X'_t)-g^{(i)}(\bm X'_0)
    =&\int_{0+}^t\nabla g^{(i)}(\bm X'_{s-})\cdot\d (\bm X')^c_s+\frac12\int_{0+}^t\nabla^2 g^{(i)}(\bm X_{s-}):\d [\bm X',\bm X']^c_s
    \\
    &\;\;+\sum_{0<s\leq t}\Big\{g^{(i)}(\bm X'_s)-g^{(i)}(\bm X'_{s-})\Big\},\text{ for }t\in[0,T].
    \end{aligned}
\end{equation*}
Therefore, by \eqref{cpresult3} of Theorem \ref{thm:condprocess}, $\d \bm Z_s = \d \big(\E[g^{(1)}(\bm X_s)|\G],\cdots,\E[g^{(n)}(\bm X_s)|\G]\big)$, and 
\begin{equation}\label{equj1}
    \begin{aligned}
J_1=\overline\E\Bigg[\int_{0+}^t&\sum_{i=1}^n\partial_{z_i}f(\bm Z_{s-},\bm Y_{s-})\bigg(\nabla g^{(i)}(\bm X'_{s-})\cdot\d (\bm X')^c_s+\frac12\nabla^2 g^{(i)}(\bm X'_{s-}):\d [\bm X',\bm X']^c_s\bigg)
    \\
    &+\sum_{0<s\leq t}\sum_{i=1}^n\partial_{z_i}f(\bm Z_{s-},\bm Y_{s-})\Big\{g^{(i)}(\bm X'_s)-g^{(i)}(\bm X'_{s-})\Big\}\bigg|\F\Bigg].
    \end{aligned}
\end{equation}
For $J_2$,  we see that
\begin{equation*}
    \begin{aligned}
    \big[&g^{(i)}(\bm X'_t),g^{(j)}(\bm X''_t)\big]_t=\big[(g^{(i)}(\bm X'))^c,(g^{j}(\bm X''))^c\big]_t+\sum_{0< s\leq t}\Delta g^{(i)}(\bm X'_s)\Delta g^{(j)}(\bm X''_s)
    \\
    &=\int_{0+}^t\Big(\nabla g^{(i)}(\bm X'_{s-})^\top\nabla g^{(j)}(\bm X''_{s-})\Big):\d [\bm X',\bm X'']^c_s
    \\
    &\qquad+\sum_{0<s\leq t}\Big\{\big(g^{(i)}(\bm X'_s)-g^{(i)}(\bm X'_{s-})\big)\big(g^{(j)}(\bm X''_s)-g^{(j)}(\bm X''_{s-})\big)\Big\},
    \end{aligned}
\end{equation*}
and 
\begin{equation*}
    \begin{aligned}
    \big[&g^{(i)}(\bm X'_t),Y^{(j)}\big]_t=\big[(g^{(i)}(\bm X'))^c,(Y^{j})^c\big]_t+\sum_{0< s\leq t}\Delta g^{(i)}(\bm X'_s)\Delta Y^{(j)}_s
    \\
    &=\int_{0+}^t\nabla g^{(i)}(\bm X'_{s-})\cdot\d [\bm X',Y^{(j)}]^c_s+\sum_{0<s\leq t}\Big\{\big(g^{(i)}(\bm X'_s)-g^{(i)}(\bm X'_{s-})\big)\big(Y^{(k)}_s-Y^{(k)}_{s-}\big)\Big\}.
    \end{aligned}
\end{equation*}
Consequently,
by \eqref{cpresult4} of Theorem \ref{thm:condprocess}, 
\begin{equation}\label{equj2}
    \begin{aligned}
    J_2=\overline\E\Bigg[\int_{0+}^t&\bigg(\frac12\sum_{i,j=1}^n\partial^2_{z_iz_j}f(\bm Z_{s-},\bm Y_{s-})\Big(\big(\nabla g^{(i)}(\bm X'_{s-})\nabla g^{(j)}(\bm X''_{s-})^\top\big):\d [\bm X',\bm X'']^c_s\Big)    \\&+\sum_{i=1}^n\sum_{k=1}^d\partial^2_{z_iy_k}f(\bm Z_{s-},\bm Y_{s-})\Big(\nabla g^{(i)}(\bm X'_{s-})\cdot\d [\bm X',Y^{(k)}]^c_s\Big)\bigg)    
    \\
    +\sum_{0<s\leq t}&\bigg(\frac12\sum_{i,j=1}^n\partial^2_{z_iz_j}f(\bm Z_{s-},\bm Y_{s-})\Big\{\big(g^{(i)}(\bm X'_s)-g^{(i)}(\bm X'_{s-})\big)\big(g^{(j)}(\bm X'_s)-g^{(j)}(\bm X'_{s-})\big)\Big\}    \\&+\sum_{i=1}^n\sum_{k=1}^d\partial^2_{z_iy_k}f(\bm Z_{s-},\bm Y_{s-})\Big\{\big(g^{(i)}(\bm X'_s)-g^{(i)}(\bm X'_{s-})\big)\big(Y^{(k)}_s-Y^{(k)}_{s-}\big)\Big\}\bigg)\bigg|\F\Bigg].
    \end{aligned}
\end{equation}
To deal with $J_3$, note that 
\begin{equation*}
    \begin{aligned}
    \Delta Z^{(i)}_t&=\big\langle\mu_t,g^{(i)}\big\rangle-\lim_{s\nearrow t}\big\langle\mu_s,g^{(i)}\big\rangle=\overline\E\big[g^{(i)}(\bm X_t)\big|\G_T\big]-\lim_{s\nearrow t}\overline\E\big[g^{(i)}(\bm X_s)\big|\G_T\big]
    \\
    &=\overline\E\big[g^{(i)}(\bm X_t)\big|\G_T\big]-\overline\E\big[\lim_{s\nearrow t}g^{(i)}(\bm X_s)\big|\G_T\big]=\overline\E\big[g^{(i)}(\bm X_t)-g^{(i)}(\bm X_{t-})\big|\G_T\big]
    \\
    &=\overline\E\big[g^{(i)}(\bm X'_t)-g^{(i)}(\bm X'_{t-})\big|\F\big],
    \end{aligned}
\end{equation*}
where the third equality holds by the  conditional dominated convergence theorem (Lemma \ref{lem:CDCT}). Therefore, we can rewrite $J_3$ as
\begin{equation}\label{equj3}
    \begin{aligned}
    J_3=\sum_{0<s\leq t}\bigg\{\overline\E\bigg[&\frac12\sum_{i,j=1}^n\partial^2_{z_iz_j} f(\bm Y_{s-},\bm Z_{s-})\big(g^{(i)}(\bm X'_s)-g^{(i)}(\bm X'_{s-})\big)\big(g^{(j)}(\bm X''_s)-g^{(j)}(\bm X''_{s-})\big)
    \\
    &+\sum_{i=1}^n\partial_{z_i} f(\bm Y_{s-},\bm Z_{s-})\big(g^{(i)}(\bm X'_s)-g^{(i)}(\bm X'_{s-})\big)\bigg|\F\bigg]\bigg\}
    \\
    +\sum_{0<s\leq t}\overline\E\bigg[&\sum_{i=1}^n\sum_{k=1}^d\partial^2_{z_iy_k} f(\bm Y_{s-},\bm Z_{s-})\big(g^{(i)}(\bm X'_s)-g^{(i)}(\bm X'_{s-})\big)\big(Y^{(k)}_s-Y^{(k)}_{s-}\big)\bigg|\F\Bigg],
    \end{aligned}
\end{equation}
since $\partial^2_{z_iz_j} f(\bm Y_{s-},\bm Z_{s-})$, $\partial_{z_i} f(\bm Y_{s-},\bm Z_{s-})$, $\partial^2_{z_iy_k} f(\bm Y_{s-},\bm Z_{s-})$, $Y^{(k)}_s$ and $Y^{(k)}_{s-}$ are $\F$-measurable random variables for $s\in(0,t]$, $1\leq i,j\leq n$, and $1\leq k\leq d$.

Recall that $
    \Phi(\mu,\bm y)=f\big(\langle \mu,g^{(1)}\rangle,..\langle \mu,g^{(n)}\rangle,\bm y\big)$,
simple calculation of linear derivatives yields a version of the derivatives by
\begin{equation*} 
    \begin{aligned}
    \nabla_{\bm y}\Phi(\mu,\bm y)=\nabla_{\bm y}f\big(\langle \mu,g^{(1)}\rangle,..\langle \mu,g^{(n)},\bm y\rangle\big),\quad\nabla^2_0\Phi(\mu,\bm y)=\nabla^2_{\bm y\bm y}f\big(\langle \mu,g^{(1)}\rangle,..\langle \mu,g^{(n)}\rangle,\bm y\big),
    \\
    \frac{\delta \Phi}{\delta\mu}(\mu,\bm y,\bm x_1)=\sum_{i=1}^n\partial_{z_i}f\big(\langle \mu,g^{(1)}\rangle,..\langle \mu,g^{(n)}\rangle,\bm y\big)g^{(i)}(\bm x_1),
    \\
    \frac{\delta^2 \Phi}{(\delta\mu)^2}(\mu,\bm y,\bm x_1,\bm x_2)=\sum_{i,j=1}^n\partial_{z_iz_j}f\big(\langle \mu,g^{(1)}\rangle,..\langle \mu,g^{(n)}\rangle,\bm y\big)g^{(i)}(\bm x_1)g^{(j)}(\bm x_2).
    \end{aligned}
\end{equation*}
Therefore, we can rewrite $J_1$ as
\begin{equation}\label{equj1new}
    \begin{aligned}
    &\;\overline\E\bigg[\int_{0+}^t\bigg(\nabla_{\bm x_1}\frac{\delta \Phi}{\delta\mu}\big(\mu_{s-},\bm Y_{s-},\bm X'_{s-}\big)\cdot\d (\bm X')^c_s+\frac12\nabla^2_{\bm x_1\bm x_1}\frac{\delta \Phi}{\delta\mu}\big(\mu_{s-},\bm Y_{s},\bm X'_s\big):\d [\bm X',\bm X']^c_t\bigg)
    \\
    &+\sum_{0<s\leq t}\bigg(\frac{\delta \Phi}{\delta\mu}\big(\mu_{s-},\bm Y_{s-},\bm X'_s\big)-\frac{\delta \Phi}{\delta\mu}\big(\mu_{s-},\bm Y_{s-},\bm X'_{s-}\big)\bigg)\bigg|\F\bigg],
    \end{aligned}
\end{equation}
 $J_2$ as 
\begin{equation}\label{equj2new}
    \begin{aligned}
    \overline\E\bigg[&\int_{0+}^t\bigg(\frac12\nabla^2_{\bm x_1\bm x_2}\frac{\delta^2 \Phi}{(\delta\mu)^2}\big(\mu_{s-},\bm Y_{s},\bm X'_s,\bm X''_s\big):\d [\bm X',\bm X'']^c_t+\nabla^2_{\bm x_1\bm y}\frac{\delta \Phi}{\delta\mu}\big(\mu_{s-},\bm Y_{s-},\bm X'_{s-}\big):\d [\bm X',\bm Y]^c_s\bigg)
    \\
    &+\sum_{0<s\leq t}\bigg(\frac12\bigg(\frac{\delta^2 \Phi}{(\delta\mu)^2}\big(\mu_{s-},\bm Y_{s-},\bm X'_{s},\bm X''_{s}\big)-\frac{\delta^2 \Phi}{(\delta\mu)^2}\big(\mu_{s-},\bm Y_{s-},\bm X'_{s-},\bm X''_{s}\big)
    \\
    &\qquad\qquad\qquad-\frac{\delta^2 \Phi}{(\delta\mu)^2}\big(\mu_{s-},\bm Y_{s-},\bm X'_{s},\bm X''_{s-}\big)+\frac{\delta^2 \Phi}{(\delta\mu)^2}\big(\mu_{s-},\bm Y_{s-},\bm X'_{s-},\bm X''_{s-}\big)\bigg)
    \\
    &\qquad\qquad\qquad+\bigg(\nabla_{\bm y}\frac{\delta \Phi}{\delta\mu}\big(\mu_{s-},\bm Y_{s-},\bm X'_{s}\big)-\nabla_{\bm y}\frac{\delta \Phi}{\delta\mu}\big(\mu_{s-},\bm Y_{s-},\bm X'_{s-}\big)\bigg)\cdot\Delta \bm Y_s\bigg)\bigg|\F\bigg],
    \end{aligned}
\end{equation}
and $J_3$ as
\begin{equation}\label{equj3new}
    \begin{aligned}
    \sum_{0<s\leq t}\bigg(\frac12\overline\E&\bigg[\frac{\delta^2 \Phi}{(\delta\mu)^2}\big(\mu_{s-},\bm Y_{s-},\bm X'_{s},\bm X''_{s}\big)-\frac{\delta^2 \Phi}{(\delta\mu)^2}\big(\mu_{s-},\bm Y_{s-},\bm X'_{s-},\bm X''_{s}\big)
    \\
    &-\frac{\delta^2 \Phi}{(\delta\mu)^2}\big(\mu_{s-},\bm Y_{s-},\bm X'_{s},\bm X''_{s-}\big)+\frac{\delta^2 \Phi}{(\delta\mu)^2}\big(\mu_{s-},\bm Y_{s-},\bm X'_{s-},\bm X''_{s-}\big)
    \bigg|\F\bigg]
    \\
    +&\overline\E\bigg[\bigg(\nabla_{\bm y}\frac{\delta \Phi}{\delta\mu}\big(\mu_{s-},\bm Y_{s-},\bm X'_{s}\big)-\nabla_{\bm y}\frac{\delta \Phi}{\delta\mu}\big(\mu_{s-},\bm Y_{s-},\bm X'_{s-}\big)\bigg)\cdot\Delta \bm Y_s\bigg|\F\bigg]
    \\
    +&\overline\E\bigg[\frac{\delta \Phi}{\delta\mu}\big(\mu_{s-},\bm Y_{s-},\bm X'_{s}\big)-\frac{\delta \Phi}{\delta\mu}\big(\mu_{s-},\bm Y_{s-},\bm X'_{s-}\big)\bigg|\F\bigg]\bigg)
    .
\end{aligned}
\end{equation}
Using the  argument at the end of \cite[Section 3.2.1]{guo2023ito}, one can cancel out redundant terms in $J_2$ and $J_3$ to get the first summation in the right hand side of \eqref{mainito}. Therefore, combining \eqref{equj1new}, \eqref{equj2new}, and \eqref{equj3new} and plugging them back into \eqref{equitof} completes the proof.


\subsection{Step 2: localization argument.}\label{subsec:local}
Take a sequence of $\FF$-stopping time $\{\tau_n\}_{n\in\N}$ such that $\tau_n\to T$, in order to prove  \eqref{mainito}  for $\Phi\in C^{2,2}(\PP_p(\R^d)\times\R^d)$ with the process $(\mu,\bm Y)=(\mu_{t\wedge\tau_n},\bm
Y_{t\wedge\tau_n})_{t\in[0,T]}$, it suffices to prove  \eqref{mainito}  for $\Phi$ with the truncated process $(\mu^{(\tau_n)},\bm Y^{(\tau_n)}):=(\mu_{t\wedge\tau_n},\bm Y_{t\wedge\tau_n})_{t\in[0,T]}$. 
Indeed, suppose \eqref{mainito} holds for the truncated process, for $t<T$. Let $n\to\infty$, it is not hard to see that the left hand side of \eqref{mainito} converges and the last three terms of the right hand side of \eqref{mainito} converge to the desired equation. 

Now, let us split the conditional expectation term of the right hand side of \eqref{mainito} into several terms. First, in order to prove that 
\begin{equation}\label{equ:firstlim}
    \overline\E\bigg[\int_{0+}^{t\wedge \tau_n}\nabla_{\bm x_1}\frac{\delta \Phi}{\delta\mu}\big(\mu_{s-},\bm Y_{s-},\bm X'_{s-}\big)\cdot\d (\bm X')^c_s\bigg|\F\bigg]\to\overline\E\bigg[\int_{0+}^{t}\nabla_{\bm x_1}\frac{\delta \Phi}{\delta\mu}\big(\mu_{s-},\bm Y_{s-},\bm X'_{s-}\big)\cdot\d (\bm X')^c_s\bigg|\F\bigg],
\end{equation}
a.s. as $n\to\infty$, we separate it by,
\begin{equation}\label{equ:temp6}
    \begin{aligned}
    \int_{0+}^{t\wedge \tau_n}&\nabla_{\bm x_1}\frac{\delta \Phi}{\delta\mu}\big(\mu_{s-},\bm Y_{s-},\bm X'_{s-}\big)\cdot\d (\bm X')^c_s=
    \\
    &\int_{0+}^{t\wedge\tau_n}\nabla_{\bm x_1}\frac{\delta \Phi}{\delta\mu}\big(\mu_{s-},\bm Y_{s-},\bm X'_{s-}\big)\cdot\d \bm X'_s-\sum_{0<s\leq t\wedge\tau_n}\nabla_{\bm x_1}\frac{\delta \Phi}{\delta\mu}\big(\mu_{s-},\bm Y_{s-},\bm X'_{s-}\big)\cdot\Delta \bm X'_s. 
    \end{aligned}
\end{equation}
Note that the first term is uniformly bounded by
\begin{equation*}
    \sup_{0\leq s\leq T}\bigg|\int_{0+}^{s}\nabla_{\bm x_1}\frac{\delta \Phi}{\delta\mu}\big(\mu_{u-},\bm Y_{u-},\bm X'_{u-}\big)\cdot\d \bm X'_u\bigg|,
\end{equation*}
and
\begin{equation*}
    \begin{aligned}
    \overline\E\bigg[\sup_{0\leq s\leq T}\bigg|\int_{0+}^{s}&\nabla_{\bm x_1}\frac{\delta \Phi}{\delta\mu}\big(\mu_{u-},\bm Y_{u-},\bm X'_{u-}\big)\cdot\d \bm X'_u\bigg|\bigg]=\bigg\|\int_{0+}^{\cdot}\nabla_{\bm x_1}\frac{\delta \Phi}{\delta\mu}\big(\mu_{u-},\bm Y_{u-},\bm X'_{u-}\big)\cdot\d \bm X'_u\bigg\|_{\S_1}
    \\
    &\leq c_1\bigg\|\int_{0+}^{\cdot}\nabla_{\bm x_1}\frac{\delta \Phi}{\delta\mu}\big(\mu_{u-},\bm Y_{u-},\bm X'_{u-}\big)\cdot\d \bm X'_u\bigg\|_{\H_1}
    \\
    &\leq c_1\bigg\|\nabla_{\bm x_1}\frac{\delta \Phi}{\delta\mu}\big(\mu_{\cdot},\bm Y_{\cdot},\bm X'_{\cdot}\big)\bigg\|_{\S_{p'}}\big\|\bm X'\big\|_{\H_p}\leq c_1C\big\||\bm X|^{p-1}+|\bm Y|^{p-1}+1\big\|_{\S_{p'}}\big\|\bm X\big\|_{\H_p}
    \\
    &\leq c_pc_1C\big(\big\|\bm X\big\|^p_{\H_p}+\big\|\bm Y\big\|^p_{\H_p}+1\big)\big\|\bm X\big\|_{\H_p}<\infty,
    \end{aligned}
\end{equation*}
where $p'=\frac{p}{p-1}$, $c_1$ and $c_p$ are constants in Proposition \ref{prop:sleqh}, and $C$ is a generic constant that may vary line by line. The first inequality is due to Proposition \ref{prop:sleqh}, the second inequality is by Proposition \ref{prop:emery}, the third inequality is from $\Phi\in C^{2,2}(\PP_p(\R^d)\times\R^d)$, and the last one holds because $\frac1p+\frac1{p'}=1$. Meanwhile, the second term of the right hand side of \eqref{equ:temp6} is bounded by
\begin{equation}\label{bd1}
    \begin{aligned}
    \sup_{0\leq s\leq T}\bigg|\nabla_{\bm x_1}\frac{\delta \Phi}{\delta\mu}\big(\mu_{s-},&\bm Y_{s-},\bm X'_{s-}\big)\bigg|\cdot\sum_{0<s\leq T}|\Delta \bm X'_s|\leq C \sup_{0\leq s\leq T}\big(1+|\bm X_s|^{p-1}+|\bm Y_s|^{p-1}\big)\cdot\sum_{0<s\leq T}|\Delta \bm X'_s|
    \\
    &\leq C\bigg(1+\sup_{0\leq s\leq T}|\bm X_s|^{p}+\sup_{0\leq s\leq T}|\bm Y_s|^{p}+\bigg(\sum_{0<s\leq T}|\Delta \bm X'_s|\bigg)^p\bigg), 
    \end{aligned}
\end{equation}
where $C$ is a generic constant, and the right hand side of the above inequality is integrable because $\|\bm X\|_{\S_p}\leq c_p\|\bm X\|_{\H_p}<\infty$, $\|\bm Y\|_{\S_p}\leq c_p\|\bm Y\|_{\H_p}<\infty$, and by the assumption of $\bm X$. Therefore, by the conditional dominated convergence theorem
{(see Lemma \ref{lem:CDCT})}, we conclude \eqref{equ:firstlim}. 

Next, to prove that 
\begin{equation}\label{equ:secondlim}
    \begin{aligned}
\overline\E\bigg[\int_{0+}^{t\wedge\tau_n}&\nabla^2_{\bm x_1\bm x_1}\frac{\delta \Phi}{\delta\mu}\big(\mu_{s-},\bm Y_{s-},\bm X'_{s-}\big):\d [\bm X',\bm X']^c_s\bigg|\F\bigg]
\\
&\to \overline\E\bigg[\int_{0+}^{t}\nabla^2_{\bm x_1\bm x_1}\frac{\delta \Phi}{\delta\mu}\big(\mu_{s-},\bm Y_{s-},\bm X'_{s-}\big):\d [\bm X',\bm X']^c_s\bigg|\F\bigg],
\end{aligned}
\end{equation}
a.s. as $n\to\infty$,  note that 
\begin{equation*}
    \begin{aligned}
        &\int_{0+}^{t\wedge\tau_n}\nabla^2_{\bm x_1\bm x_1}\frac{\delta \Phi}{\delta\mu}\big(\mu_{s-},\bm Y_{s-},\bm X'_{s-}\big):\d [\bm X',\bm X']^c_s
    \end{aligned}
\end{equation*} 
is bounded by
\begin{equation*}
    \sup_{0\leq s\leq T}\bigg|\int_{0+}^{s}\nabla^2_{\bm x_1\bm x_1}\frac{\delta \Phi}{\delta\mu}\big(\mu_{u-},\bm Y_{u-},\bm X'_{u-}\big):\d [\bm X',\bm X']^c_u\bigg|\leq\sup_{0\leq s\leq T}\bigg|\nabla^2_{\bm x_1\bm x_1}\frac{\delta \Phi}{\delta\mu}\big(\mu_{s-},\bm Y_{s-},\bm X'_{s-}\big)\bigg|\cdot [\bm X',\bm X']^c_T,
\end{equation*}
and
\begin{equation*}
    \begin{aligned}
    \overline \E\bigg[\sup_{0\leq s\leq T}\bigg|\nabla^2_{\bm x_1\bm x_1}\frac{\delta \Phi}{\delta\mu}\big(\mu_{s-},\bm Y_{s-},\bm X'_{s-}\big)&\bigg|\cdot [\bm X',\bm X']^c_T\bigg]\leq \bigg\|\nabla^2_{\bm x_1\bm x_1}\frac{\delta \Phi}{\delta\mu}\big(\mu_{\cdot},\bm Y_{\cdot},\bm X'_{\cdot}\big)\bigg\|_{\S_{p''}}\big\|[\bm X',\bm X']^c_T\big\|_{L_{p/2}}
    \\
    &\leq C\Big\|1+|\bm X'|^{p-2}+|\bm Y|^{p-2}\Big\|_{\S_{p''}}\big\|[\bm X',\bm X']^c_T\big\|_{L_{p/2}}
    \\
    &\leq c_{p''}C\Big(1+\|\bm X'\|^{p}_{\H_{p}}+\|\bm Y\|^{p}_{\H_{p}}\Big)\big\|[\bm X',\bm X']^c_T\big\|_{L_{p/2}}<\infty,
    \end{aligned}
\end{equation*}
where $c_{p''}$ is the constant in Proposition \ref{prop:sleqh}, $C$ is a generic constant that may vary from line to line and $p''=\frac{p-2}{2p}$. Thus,  \eqref{equ:secondlim} holds by Lemma \ref{lem:CDCT}. Similar arguments holds also for the convergence of $\int_{0+}^{t\wedge \tau_n}\nabla^2_{\bm x_1\bm x_2}\frac{\delta^2 \Phi}{(\delta\mu)^2}\big(\mu_{s-},\bm Y_{s-},\bm X'_{s-},\bm X''_{s-}\big):\d [\bm X',\bm X'']_s$ and $\int_{0+}^{t\wedge \tau_n}\nabla^2_{\bm x_1\bm y}\frac{\delta \Phi}{\delta\mu}\big(\mu_{s-},\bm Y_{s-},\bm X'_{s-}\big):\d [\bm X',\bm Y]_s$.

Next, let us deal with the convergence of
\begin{equation} \label{eqtem4} 
    \begin{aligned}
\sum_{0<s\leq t\wedge \tau_n}\bigg(&\frac{\delta^2 \Phi}{(\delta\mu)^2}\big(\mu_{s-},\bm Y_{s-},\bm X'_{s},\bm X''_{s}\big)-\frac{\delta^2 \Phi}{(\delta\mu)^2}\big(\mu_{s-},\bm Y_{s-},\bm X'_{s-},\bm X''_{s}\big)
\\
&-\frac{\delta^2 \Phi}{(\delta\mu)^2}\big(\mu_{s-},\bm Y_{s-},\bm X'_{s},\bm X''_{s-}\big)+\frac{\delta^2 \Phi}{(\delta\mu)^2}\big(\mu_{s-},\bm Y_{s-},\bm X'_{s-},\bm X''_{s-}\big)\bigg)\mathbbm {1}_{\{\mu_s=\mu_{s-}\}}.
\end{aligned}
\end{equation}
Applying  the fundamental theorem of calculus for the linear derivative, we have 
\begin{equation*}  
    \begin{aligned}
&\frac{\delta^2 \Phi}{(\delta\mu)^2}\big(\mu_{s-},\bm Y_{s-},\bm X'_{s},\bm X''_{s}\big)-\frac{\delta^2 \Phi}{(\delta\mu)^2}\big(\mu_{s-},\bm Y_{s-},\bm X'_{s-},\bm X''_{s}\big)
\\
&-\frac{\delta^2 \Phi}{(\delta\mu)^2}\big(\mu_{s-},\bm Y_{s-},\bm X'_{s},\bm X''_{s-}\big)+\frac{\delta^2 \Phi}{(\delta\mu)^2}\big(\mu_{s-},\bm Y_{s-},\bm X'_{s-},\bm X''_{s-}\big)
\\
=\int_0^1&\int_0^1\nabla^2_{\bm x_1\bm x_2}\frac{\delta^2 \Phi}{(\delta\mu)^2}\big(\mu_{s-},\bm Y_{s-},\bm X'_{s-}+\lambda(\bm X'_{s}-\bm X'_{s-}),\bm X''_{s-}+\gamma(\bm X''_{s}-\bm X''_{s-})\big):\Big(\big(\Delta \bm X'_s\big)^\top\Delta \bm X''_s\Big)\d \lambda\d\gamma,
\end{aligned}
\end{equation*}
and the right hand side of the above equation  is bounded by
\begin{equation*}
    C\cdot |\Delta \bm X'_s|\cdot |\Delta \bm X''_s|\cdot \Big(1+|\bm X'_s|^{p-2}+|\bm X''_s|^{p-2}\Big).
\end{equation*}
Therefore, \eqref{eqtem4} is bounded by 
\begin{equation*}
    \begin{aligned}
    &\sup_{0\leq s\leq T}\Big(1+|\bm X'_s|^{p-2}+|\bm X''_s|^{p-2}\Big)\big(\sum_{0<s\leq T}|\Delta \bm X'_s|\cdot |\Delta \bm X''_s|\big)
    \\
    &\leq C\bigg(\sup_{0\leq t\leq T}\Big(1+|\bm X'_t|^{p-2}+|\bm X''_t|^{p-2}\Big)^{\frac{p}{p-2}}+\Big(\sum_{0<s\leq T}|\Delta \bm X'_s|^2+\sum_{0<s\leq T}|\Delta \bm X''_s|^2\Big)^{\frac{p}{2}}\bigg)
    \\
    &\leq C\bigg(\sup_{0\leq t\leq T}\Big(1+|\bm X'_t|^{p}+|\bm X''_t|^{p}\Big)+\sum_{0<s\leq T}|\Delta \bm X'_s|^{p}+\sum_{0<s\leq T}|\Delta \bm X''_s|^{p}\bigg), 
    \end{aligned}
\end{equation*}
where $C$ is a generic constant. The right hand side of the above inequality is integrable and the convergence of the conditional expectation of \eqref{eqtem4} is ensured by Lemma \ref{lem:CDCT}. Similar arguments can be applied to deduce the convergence of the conditional expectation of $\big(\nabla_{\bm y}\frac{\delta \Phi}{\delta\mu}\big(\mu_{s-},\bm Y_{s-},\bm X'_{s}\big)-\nabla_{\bm y}\frac{\delta \Phi}{\delta\mu}\big(\mu_{s-},\bm Y_{s-},\bm X'_{s-}\big)\big)\cdot\Delta \bm Y_s$ and $\frac{\delta \Phi}{\delta\mu}\big(\mu_{s-},\bm Y_{s-},\bm X'_s\big)-\frac{\delta \Phi}{\delta\mu}\big(\mu_{s-},\bm Y_{s-},\bm X'_{s-}\big)$,  and thus we conclude that \eqref{mainito} holds for the original process $(\mu,\bm Y)$ for $0\leq t<T$. 

For the case of  $t=T$, taking $n\to\infty$, then the left hand side equals $\Phi(\mu_{T-},\bm Y_{T-})-\Phi(\mu_0,\bm Y_0)$ while all terms on the right hand side are  evaluated at $T$ except for the last term which is 
\begin{equation*}
    \sum_{0<s<t}\Big(\Phi(\mu_s,\bm Y_s)-\Phi(\mu_{s-},\bm Y_{s-})\Big),
\end{equation*}
adding $\Phi(\mu_{T},\bm Y_{T})-\Phi(\mu_{T-},\bm Y_{T-})$ to both sides, we finish the  proof of the claim.

\subsection{Step 3: cylindrical functions are dense in $C^{2,2}(\R^d\times\mathcal{P}_p(\R^d))$.}\label{subsec:dense}
The following result slightly generalizes that in \cite[Theorem 4.4, Theorem 4.10, Section 4]{cox2024controlled} as it also includes the convergences of the associated linear derivatives.
\begin{proposition}\label{dense_cylindrical}
Let $\Phi \in C^{2,2}(\PP_p(\R^d)\times \R^d)$, then there exists a sequence of  $C^{2,2}$ cylinder functions $\{f_n\}_{n\in\N^+}$  such that one has the point-wise convergence 
\begin{equation}\label{cylin_converge}
    \begin{aligned}
    \bigg(f_n, \nabla_{\bm y} f_n&, \nabla_{\bm y}^2 f_n, \frac{\delta f_n}{\delta \mu}, \nabla_{\bm x_1}\frac{\delta f_n}{\delta\mu},\nabla^2_{\bm x_1}\frac{\delta f_n}{\delta\mu},\nabla_{\bm x_1\bm y}\frac{\delta f_n}{\delta\mu},\nabla_{\bm x_1\bm x_2}\frac{\delta^2 f_n}{(\delta \mu)^2}\bigg)
    \\
    &\to \bigg(\Phi, \nabla_{\bm y} \Phi, \nabla_{\bm y}^2 \Phi, \frac{\delta \Phi}{\delta \mu}, \nabla_{\bm x_1}\frac{\delta \Phi}{\delta\mu},\nabla^2_{\bm x_1}\frac{\delta \Phi}{\delta\mu},\nabla_{\bm x_1\bm y}\frac{\delta \Phi}{\delta\mu},\nabla_{\bm x_1\bm x_2}\frac{\delta^2 \Phi}{(\delta \mu)^2}\bigg)
    \end{aligned}
\end{equation}
as $n\to\infty$,  for all $\bm x_1,\bm x_2,\bm y\in\R^d$, $\mu\in \PP_p(\R^d)$. Moreover, there exists a constant $c>0$ such
that
\begin{equation}\label{cylin_boun}
    \begin{aligned}
    |\nabla_{\bm y} f_n|(\mu,\bm y)\leq c(1+|\bm y|^{p-1}),\qquad|\nabla^2_{\bm y} f_n|(\mu,\bm y)&\leq c(1+|\bm y|^{p-2}),
    \\
    \bigg|\frac{\delta f_n}{\delta \mu}\bigg|\leq c\big(1+|\bm x_1|^{p}+|\bm y|^{p}\big),\qquad \bigg|\nabla_{\bm x_1}\frac{\delta f_n}{\delta\mu}\bigg|&\leq c\big(1+|\bm x_1|^{p-1}+|\bm y|^{p-1}\big),
    \\
    \bigg|\nabla^2_{\bm x_1}\frac{\delta f_n}{\delta\mu}\bigg|\leq c\big(1+|\bm x_1|^{p-2}+|\bm y|^{p-2}\big),\qquad\bigg|\nabla^2_{\bm x_1\bm y}\frac{\delta f_n}{\delta\mu}\bigg|&\leq c\big(1+|\bm x_1|^{{p-2}}+|\bm y|^{p-2}\big),
    \\
    \bigg|\nabla^2_{\bm x_1\bm x_2}\frac{\delta^2 f_n}{(\delta \mu)^2}\bigg|&\leq c\big(1+|\bm x_1|^{p-2}+|\bm x_2|^{p-2}+|\bm y|^{p-2}\big),
    \end{aligned}
\end{equation}
for all $\bm x_1,\bm x_2,\bm y\in\R^d$, $\mu\in \PP_p(\R^d)$.
\end{proposition}
The proof is similar to that of \cite[Section 4]{cox2024controlled}. {  Following their construction procedure, fix $n \in \N$ and cover the ball $B_n := \{\bm x \in \R^d: |\bm x| \leq n\} $ by finitely many open sets of diameter at most $1/n$, denoted by $U^n_i$, with $i = 1,...,N_n$. Set $U^n_0 =\R^d\setminus B_n$ to get an open cover in $\R^d$. Then, fix points ${\bm x}^n_i$ in $\overline {U^n_i}$ with the minimal norm for $i=0,1,..,N_n$. Next, define a family of functions $\{\psi^n_ i\}$ as a partition of unity subordinate to $\{U^n_i\}$. (That is, each $\psi^n_i$ is a continuous function supported on $U^n_i$, such that $\sum_{i=0}^{N_n}\psi^n_i(\bm x) = 1$ for all $\bm x \in \R^d$). Finally, define the operator $T_n:C_b(\R^d)\to C_b(\R^d)$ such that for any function $\varphi$ on $\R^d$,
$$
T_n\varphi(x)=\sum_{i=0}^{N_n}\varphi({\bm x}^n_i)\psi^n_i(\bm x).
$$
Its adjoint operator $T_n^*:\PP_p(\R^d)\to \PP_p(\R^d)$ and the tensor square operator $T_n^{\otimes 2}:C_b(\R^{2d})\to C_b(\R^{2d})$ are defined accordingly.} 
Now, define $f_n(\mu,\bm y):=\Phi(T_n^*\mu,\bm y)$, then $f_n$ are $C^{2,2}$ cylindrical functions  and the key identities to be established are 
\begin{equation*}
    \frac{\delta f_n}{\delta \mu}(\mu,\bm y,\bm x_1)=T_n\Big(\frac{\delta \Phi}{\delta \mu}(T_n^*\mu,\bm y,\cdot)\Big)(\bm x_1),\quad\frac{\delta^2 f_n}{(\delta \mu)^2}(\mu,\bm y,\bm x_1,\bm x_2)=T_n^{\otimes2}\frac{\delta^2 \Phi}{(\delta\mu)^2}(T_n^*\mu, \bm y, \cdot,\cdot)(\bm x_1,\bm x_2),
\end{equation*}
for all $\bm x_1,\bm x_2,\bm y\in\R^d$ and $\mu\in\PP_p(\R^d)$ (see \cite[Lemma 4.6, Lemma 4.12]{cox2024controlled}).
To this end, note that the construction of $T_n$ ensures
\begin{equation*}
    T_n\nabla_{\bm x_1}\varphi(\bm x_1)=\nabla_{\bm x_1}T_n\varphi(\bm x_1),\quad T_n^{\otimes 2}\nabla^2_{\bm x_1\bm x_2}\psi(\bm x_1,\bm x_2)=\nabla^2_{\bm x_1\bm x_2}T_n\psi(\bm x_1,\bm x_2),
\end{equation*}
for all $\varphi\in C^1_b(\R^d)$ and $\psi\in C^2_b(\R^{2d}).$ Thus repeating the proof of \cite[Theorem 4.4, 
Theorem 4.10]{cox2024controlled} finishes the proof of the proposition.


\subsection{Final step} Define $\tau_k:=\inf\{t\in[0,T]:|\bm Y_t|> k\}$, then by Step 1,  It\^o's formula \eqref{mainito} holds for the truncated process $(\mu^{(\tau_n)},\bm Y^{(\tau_n)}):=(\mu_{t\wedge\tau_n},\bm Y_{t\wedge\tau_n})_{t\in[0,T]}$ with the sequences of cylindrical functions $\{f_n\}_{n\in\N^+}$. By Step 2, the localization argument shows that It\^o's formula \eqref{mainito} holds for the sequences of cylindrical functions $\{f_n\}_{n\in\N^+}$. 

In order to pass the limit from $f_n$ to $\Phi$, let us consider the first term in the right hand side of \eqref{mainito} and the remaining terms follow similarly. For the first term, we need to prove that there exists a subsequence $\{n_k\}_{k\in\N^+}$ such that
\begin{equation}\label{conver:condexp}
    \E\bigg[\int_{0+}^t\nabla_{\bm x_1}\frac{\delta f_{n_k}}{\delta\mu}\big(\mu_{s-},\bm Y_{s-},\bm X'_{s-}\big)\cdot\d (\bm X')^c_s\Big|\F\bigg]\to\E\bigg[\int_{0+}^t\nabla_{\bm x_1}\frac{\delta \Phi}{\delta\mu}\big(\mu_{s-},\bm Y_{s-},\bm X'_{s-}\big)\cdot\d (\bm X')^c_s\Big|\F\bigg],
\end{equation}
as $k\to\infty.$ To this end, we first prove that 
\begin{equation}\label{conver:stoint}
\int_{0+}^t\nabla_{\bm x_1}\frac{\delta f_n}{\delta\mu}\big(\mu_{s-},\bm Y_{s-},\bm X'_{s-}\big)\cdot\d (\bm X')^c_s\overset{L^1}{\to}\int_{0+}^t\nabla_{\bm x_1}\frac{\delta \Phi}{\delta\mu}\big(\mu_{s-},\bm Y_{s-},\bm X'_{s-}\big)\cdot\d (\bm X')^c_s,
\end{equation}
as $n\to\infty$. In fact, by Proposition \ref{dense_cylindrical}, we have
\begin{equation*}
\bigg||\nabla_{\bm x_1}\frac{\delta f_n}{\delta\mu}\big(\mu_{s-},\bm Y_{s-},\bm X'_{s-}\big)\bigg|\leq c\big(1+|\bm X_{s-}|^{p-1}+|\bm Y_{s-}|^{p-1}\big),
\end{equation*}
where the $\S_{p'}$ norm of the right hand side is finite with $p'=\frac{p}{p-1}$.
Now \eqref{conver:stoint} follows from $\|\bm X\|_{\H^p}<\infty$ and Lemma \ref{lem:CDCT}.

Next, since
\begin{equation*}
    \begin{aligned}
    \E\bigg[\Big|& \E\Big[\int_{0+}^t\nabla_{\bm x_1}\frac{\delta f_{n}}{\delta\mu}\big(\mu_{s-},\bm Y_{s-},\bm X'_{s-}\big)\cdot\d (\bm X')^c_s\Big|\F\Big]- \E\Big[\int_{0+}^t\nabla_{\bm x_1}\frac{\delta \Phi}{\delta\mu}\big(\mu_{s-},\bm Y_{s-},\bm X'_{s-}\big)\cdot\d (\bm X')^c_s\Big|\F\Big]\Big|\bigg]
    \\
    &\leq \E\bigg[\Big| \int_{0+}^t\nabla_{\bm x_1}\frac{\delta f_{n}}{\delta\mu}\big(\mu_{s-},\bm Y_{s-},\bm X'_{s-}\big)\cdot\d (\bm X')^c_s- \int_{0+}^t\nabla_{\bm x_1}\frac{\delta \Phi}{\delta\mu}\big(\mu_{s-},\bm Y_{s-},\bm X'_{s-}\big)\cdot\d (\bm X')^c_s\Big|\bigg]\to 0,
    \end{aligned}
\end{equation*}
and $L_1$ convergences implies convergence in subsequence, \eqref{conver:condexp} holds. The remaining terms follow similarly due to the uniform bound \eqref{cylin_boun}, and thus the  proof of  Theorem \ref{mainthm} is complete.

\section{Proof of Results in Section \ref{section:condcopy}}\label{sec:proofcondcopy}
\subsection{Proof of Theorem \ref{thm:condindepcopy}}\label{sec:conconstruct}
To prove the theorem, we first provide a constructive proof for  the existence of  conditionally independent copy claimed in Theorem \ref{thm:condindepcopy}. 
For ease of exposition,  we will show how to enlarge the probability space with two independent copies, and the case with a finite number of copies can be easily adapted. 

Recall the discussions in Section \ref{section:condcopy}, and consider the probability space $\overline\Omega$, the $\sigma$-algebra $\overline\F$ defined by
\begin{equation}\label{sigmafextension}
\overline\Omega=\Omega^3=\{(\omega_0,\omega_1,\omega_2)|\omega_0,\omega_1,\omega_2\in\Omega\}, \qquad \overline\F=\sigma\{A_0\times A_1\times A_2:A_0,A_1,A_2\in\F\},
\end{equation}
and a probability measure $\overline \P$  defined as follows: for $A_0,A_1,A_2\in\F$, 
\begin{equation}\label{probextension}
    \overline \P\big(A_0\times A_1\times A_2\big):=\E\Big[\mathbbm{1}_{A_0}\P(A_1|\G)\P(A_2|\G)\Big].
\end{equation}
Since the set $\{A_0\times A_1\times A_2:A_0,A_1,A_2\in\F\}$ is a $\pi$-system,  the measure on $(\overline \Omega,\overline \F,\overline\P)$ is well defined. 

Next, take $\widetilde\F,\widetilde\G$ for the natural extension in this enlarged probability space as
\begin{equation}\label{extenfg}
    \widetilde\F:=\{A\times\Omega\times\Omega:A\in\F\}, \quad \widetilde\G:=\{A\times\Omega\times\Omega:A\in\G\},
\end{equation}
and define $\F'$, $\F''$ as
\begin{equation}\label{extenff'}
\F':=\{\Omega\times A\times\Omega:A\in\F\}, \quad\F'':=\{\Omega\times\Omega\times A:A\in\F\}.
\end{equation}

In this framework, we have the following proposition.

\begin{proposition}\label{prop:condcopy}
    Consider $(\overline \Omega,\overline \F, \overline \P)$ defined by \eqref{sigmafextension} and \eqref{probextension} and the sub-algebras $\widetilde\F,\widetilde\G$, $\F'$, and  $\F''$ defined by \eqref{extenfg} and \eqref{extenff'}. Suppose $A,B\in\F$, and let $\widetilde A=A\times\Omega\times\Omega$, $A'=\Omega\times A\times \Omega$, $A''=\Omega\times \Omega\times A$, $B'=\Omega\times B\times \Omega$, and $C''=\Omega\times \Omega\times C$. Then
    \begin{enumerate}
        \item\label{conditionexten} For $\overline\omega=(\omega_0,\omega_1,\omega_2)\in\overline\Omega$,
        \begin{equation*}
            \P(A|\mathcal G)(\omega_0)=\overline\P(\widetilde A|\widetilde \G)(\omega_0,\omega_1,\omega_2),\text { a.s. under }\overline\P.
        \end{equation*}
        
        \item\label{conditionident} Given $\widetilde\G$,  $\widetilde A,A',$ and $A''$ have  the same conditional law. That is, 
        \begin{equation*}
            \overline \P(\widetilde A|\widetilde\G)=\overline \P(A'|\widetilde\G)=\overline \P(A''|\widetilde\G),\text { a.s. under }\overline\P.
        \end{equation*}
        Moreover, this conditional probability $\overline \P(\widetilde A|\widetilde\G)$ coincides with the conditional probability of $A'$ and $A''$, given $\widetilde\F$. That is,  
        \begin{equation*}
            \overline \P(\widetilde A|\widetilde\G)=\overline \P(A'|\widetilde\F)=\overline \P(A''|\widetilde\F),\text { a.s. under }\overline\P.
        \end{equation*}
        \item\label{conditionindep} Given $\widetilde \G$,  $\widetilde \F,\F',$ and $\F''$ are conditionally independent. That is, 
        \begin{equation*}
            \overline\P\big(\widetilde A\cap B'\cap C''|\widetilde\G\big)=\overline\P\big(\widetilde A|\widetilde\G\big)\overline\P\big(B'|\widetilde\G\big)\overline\P\big(C''|\widetilde\G\big),\text { a.s. under }\overline\P.
        \end{equation*}
        Moreover, given $\widetilde \F$, then $\F'$ and $\F''$ are conditionally independent.   That is,
        \begin{equation*}
            \overline\P\big(B'\cap C''|\widetilde\F\big)=\overline\P\big(B'|\widetilde\F\big)\overline\P\big(C''|\widetilde\F\big),\text { a.s. under }\overline\P.
        \end{equation*}
    \end{enumerate}
\end{proposition}

\begin{proof} (of Proposition \ref{prop:condcopy}).
To prove \eqref{conditionexten}, note that $\P[A|\G]$ is $\G$ measurable in $(\Omega,\F,\FF,\P)$, then $\P[A|\G](\omega_1)$ as a random variable in $(\overline\Omega,\overline\F,\overline\FF,\overline\P)$ is $\widetilde\G$-measurable. For any $\widetilde G\in\widetilde \G$, there exists a $G\in\G$ such that $\widetilde G=G\times\Omega\times\Omega\in\widetilde\G$. By the definition of $\overline{\P}$ in (\eqref{probextension}) and the property of conditional expectation, 
\begin{equation*}
\overline\P\big[\widetilde A\cap\widetilde G\big]=\overline\P\big[(A\cap G)\times\Omega\times\Omega\big]=\P[A\cap G]=\E\Big[\P[A|\G]\mathbbm1_{G}\Big]=\int_\Omega\P[A|\G](\omega)\mathbbm1_{G}(\omega)\P(\d \omega).
\end{equation*}
Again by \eqref{probextension}, $\overline\P(\d \omega_1\times\Omega\times\Omega)=\P(\d \omega_1)$, therefore, 
\begin{equation}\label{eqtem7}
    \begin{aligned}
    \overline\P\big[\widetilde A\cap\widetilde G\big]&=\int_{\Omega^3}\P[A|\G](\omega_0)\mathbbm1_{G}(\omega_0)\overline\P(\d \omega_0\d\omega_1\d\omega_2)
    \\&=\int_{\Omega^3}\P[A|\G](\omega_0)\mathbbm1_{\widetilde G}(\omega_0,\omega_1,\omega_2)\overline\P(\d \omega_0\d\omega_1\d\omega_2)=\overline\E\Big[\P[A|\G]\mathbbm1_{\widetilde G}\Big],
    \end{aligned}
\end{equation}
which leads to $\P[A|\G](\omega_0)=\overline\P[\widetilde A|\widetilde\G](\omega_0,\omega_1,\omega_2)$ a.s. under $\overline P$.

To prove \eqref{conditionident}, for any $\widetilde G\in\widetilde \G$, there exists a $G\in\G$ such that $\widetilde G=G\times\Omega\times\Omega\in\widetilde\G$. Then,
\begin{equation*}
    \overline\P\big[A'\cap\widetilde G\big]=\overline\P\big[G\times A\times \Omega\big]=\E\big[\mathbbm1_G\P(A|\G)\big]=\overline\E\Big[\P[A|\G]\mathbbm1_{\widetilde G}\Big]=\overline\E\Big[\overline\P[\widetilde A|\widetilde \G]\mathbbm1_{\widetilde G}\Big],
\end{equation*}
where the second equality holds because of \eqref{probextension}, the third equality holds due to the same calculation as \eqref{eqtem7} and the last equality holds by \eqref{conditionexten}. This leads to $\overline \P(A'|\widetilde\G)=\overline \P(\widetilde A|\widetilde\G)$ and similarly, we have $\overline \P(A''|\widetilde\G)=\overline \P(\widetilde A|\widetilde\G)$. 

As for the second statement of \eqref{conditionident}, note that $\overline\P(A'|\widetilde \G)$ is $\widetilde \G$ measurable, hence $\widetilde \F$ measurable. For any $\widetilde F\in\widetilde\F$, there exists an $F\in\F$ such that $\widetilde F=F\times\Omega\times\Omega\in\widetilde\F$, and
\begin{equation*}
    \begin{aligned}
    \overline\P(A'\cap\widetilde F)=\overline\P\big[F\times A\times \Omega\big]=\E\big[\mathbbm1_F\P(A|\G)\big]=\overline\E\Big[\P[A|\G]\mathbbm1_{\widetilde F}\Big]=\overline\E\Big[\overline\P[\widetilde A|\widetilde \G]\mathbbm1_{\widetilde F}\Big],
    \end{aligned}
\end{equation*}
where the second equality follows from \eqref{probextension}, the third equality holds due to the same calculation as \eqref{eqtem7}, and the last equation is by \eqref{conditionexten}. This leads to $\overline \P(A'|\widetilde\F)=\overline \P(\widetilde A|\widetilde\G)$ and similarly,  $\overline \P(A''|\widetilde\F)=\overline \P(\widetilde A|\widetilde\G)$. 

To prove \eqref{conditionindep}, for any $\widetilde G\in\widetilde \G$, there exists a $G\in\G$ such that $\widetilde G=G\times\Omega\times\Omega\in\widetilde\G$, and  
\begin{equation*}
    \begin{aligned}
    \overline{\P}(\widetilde A\cap B'\cap C''\cap\widetilde G)&=\overline{\P}((A\cap G)\times B\times C)=\E\Big[\mathbbm1_{A\cap G}\P(B|\G)\P(C|\G)\Big]
    \\
    &=\E\Big[\E[\mathbbm1_{A\cap G}|\G]\P(B|\G)\P(C|\G)\Big]=\E\Big[\mathbbm1_{G}\P(A|\G)\P(B|\G)\P(C|\G)\Big].
    \end{aligned}
\end{equation*}
By the same calculation as in  \eqref{eqtem7} and noting that $\P(A|\G)=\overline\P(\widetilde A|\widetilde \G)$, $\P(B|\G)=\overline\P(B'|\widetilde \G)$, and $\P(C|\G)=\overline\P(C''|\widetilde \G)$, we have
    \begin{equation*}
    \begin{aligned}
    \overline{\P}(\widetilde A\cap B'\cap C''\cap\widetilde G)&=\overline\E\Big[\mathbbm1_{\widetilde G}\P(A|\G)\P(B|\G)\P(C|\G)\Big]=\overline\E\Big[\mathbbm1_{\widetilde G}\overline\P(\widetilde A|\widetilde \G)\overline\P(B'|\widetilde \G)\overline\P(C''|\widetilde \G)\Big].
    \end{aligned}
\end{equation*}
Therefore, 
\begin{equation*}
    \overline\P\big(\widetilde A\cap B'\cap C''|\widetilde\G\big)=\overline\P(\widetilde A|\widetilde \G)\overline\P(B'|\widetilde \G)\overline\P(C''|\widetilde \G),\text{ a.s. under }\overline \P.
\end{equation*}

As for the second statement of  \eqref{conditionindep}, for any $\widetilde F\in\widetilde\F$, there exists an $F\in\F$ such that $\widetilde F=F\times\Omega\times\Omega\in\widetilde\F$, and by \eqref{probextension},
\begin{equation*}   
    \begin{aligned}
    \overline\P(B'\cap C''\cap \widetilde F)&=\overline\P(F\times B\times C)=\E\big[\mathbbm1_{ F}\P(B|\G)\P(C|\G)\big]
    \\&=\overline\E\big[\mathbbm1_{\widetilde F}\P(B|\G)\P(C|\G)\big]=\overline\E\big[\mathbbm1_{\widetilde F}\overline\P(B'|\widetilde\F)\overline\P(C''|\overline\F)\big],
    \end{aligned}
\end{equation*}
where the  the third equality follows from the  same calculation as \eqref{eqtem7} and the last equation holds by \eqref{conditionexten} and \eqref{conditionident}. Therefore,  
\begin{equation*}
    \overline\P\big(B'\cap C''|\widetilde\F\big)=\overline\P\big(B'|\widetilde\F\big)\overline\P\big(C''|\widetilde\F\big),\text { a.s. under }\overline\P.
\end{equation*}
\end{proof}

\begin{proof}  (of Theorem \ref{thm:condindepcopy}). For any random variable $R$ { taking values in a standard Borel space $(E,\mathcal{E})$}, recall that the natural extension $\widetilde R$ and the two conditionally independent copies $R'$ and $R''$ of $R$ are defined by
\begin{equation*}   
    \widetilde R(\omega_0,\omega_1,\omega_2)=R(\omega_0),\; R'(\omega_0,\omega_1,\omega_2)=R(\omega_1),\;R''(\omega_0,\omega_1,\omega_2)=R(\omega_2).
\end{equation*}
Then 
\begin{equation*}
\text{Law}\big(R\big|\G\big)(\omega_0)=\text{Law}\big(\widetilde R\big|\widetilde\G\big)(\overline\omega)=\text{Law}\big(R'\big|{\widetilde\G}\big)(\overline\omega)=\text{Law}\big(R''\big|{\widetilde\F}\big)(\overline\omega)\text{ a.s. under }\overline\P,
\end{equation*}
as a direct consequence of $\eqref{conditionident}$ of Proposition \ref{prop:condcopy}, with $\overline\omega=(\omega_0,\omega_1,\omega_2)$, { since these laws can be uniquely determined from $\P(R\in A|\G)$,$\overline\P(\widetilde R\in A|\widetilde\G)$, $\overline\P(R'\in A|\widetilde\G)$, and $\overline\P(R''\in A|\widetilde\F)$ for a countable collection of sets $A$.
} Meanwhile, \eqref{conditionindep} implies that  
$\widetilde R$, $R'$, and $R''$ are conditionally independent given $\widetilde\G$ or $\widetilde\F$.
   
\end{proof}

{ 
\begin{remark} Our construction   
is inspired by the proof of Yamada-Watanabe theorem (see for instance \cite{karatzas-shreve,kurtz2007yamada}) and the theory of weak and strong solutions of stochastic differential equations (see for instance \cite{jacod1981weak,kurtz2014weak}). 
Note that in these existing approaches, the common information $\G$ is modeled as part of the stochastic input. Moreover, in the context of weak and strong solutions of SDEs, one often uses the existence of regular conditional distributions to construct an extended probability space by embedding different weak solutions into a common space. To illustrate this, take the example of  $(X^{(1)},W^{(1)})$ and $(X^{(2)},W^{(2)})$,  let $\nu(\d w)$ denote the common law of a Brownian motion, and define the regular conditional law $\eta^{(i)}(\d x;w):=\text{Law}(X^{(i)}|W^{(i)}=w)$ for $i=1,2$. In such a setting, an extended probability measure is often defined (informally) by $\P(\d x^{(1)}\d x^{(2)}\d w):=\eta^{(1)}(\d x^{(1)};w)\cdot \eta^{(2)}(\d x^{(2)};w)\nu(\d w).$ 

Motivated by these constructions, we consider a more general setting in which the common information $\G$ may be abstract and regular conditional distributions may not be available. We define conditionally independent copies of some random variable $X$ directly by specifying the measure on the extended space supporting $X,X'$ such that
$$
\P\big((X,X')\in A\times B\big):=\E[\mathbbm 1_{X\in A}\cdot \P(X\in B|\G)],
$$
which remains consistent with the spirit of the classical constructions while avoiding the explicit use of regular conditional distributions.
 
\end{remark} }

    

\noindent\textbf{Proof of Corollary \ref{cor:condindepcopy}.} It suffices to show that  the extension of the sub-filtration $\widetilde\GG\subset\widetilde\FF$ satisfies \begin{equation*}
    {\widetilde\F}_t\indep{\widetilde\G}_T\big|{\widetilde\G}_t,\text{ for all }0\leq t\leq T.
\end{equation*}
For any $\widetilde F_t\in\widetilde\F_t$, $\widetilde G_t\in{\widetilde\G}_t$ and $\widetilde G_T\in{\widetilde\G}_T$,  there exist $F_t\in\F_t$, $G_t\in\G_t$, and $G_T\in\G_T$ such that $\widetilde F_t=F_t\times\Omega\times\Omega$,  $\widetilde G_t=G_t\times\Omega\times\Omega$ and $\widetilde G_T=G_T\times\Omega\times\Omega$. Then,
\begin{equation*}
    \begin{aligned}
    \overline{\P}(\widetilde F_t\cap\widetilde G_t\cap\widetilde G_T)&=\overline \P((F_t\cap G_t\cap G_T)\times\Omega\times\Omega)=\P(F_t\cap G_t\cap G_T)
    \\
    &=\E\big[\P(F_t\cap G_T|\G_t)\mathbbm1_{G_t}\big]=\E\big[\P(F_t|\G_t)\P(G_T|\G_t)\mathbbm1_{G_t}\big]
    \\
    &=\overline\E\big[\P(F_t|\G_t)\P(G_T|\G_t)\mathbbm1_{\widetilde G_t}\big]=\overline\E\big[\overline\P(\widetilde F_t|\widetilde \G_t)\overline\P(\widetilde G_T|\widetilde \G_t)\mathbbm1_{\widetilde G_t}\big],
    \end{aligned}
\end{equation*}
where the fourth equality holds by $\F_t\indep\G_T|\G_t$, the fifth equality follows from the same calculation as \eqref{eqtem7}, and the last equality is by \eqref{conditionident} in Proposition \ref{prop:condcopy}. Therefore,
\begin{equation*}
\overline \P(\widetilde F_t\cap\widetilde G_T|\widetilde\G_t)=\overline \P(\widetilde F_t|\widetilde\G_t)\overline \P(\widetilde G_T|\widetilde\G_t).
\end{equation*}
\qed

   

\subsection{Proof of Proposition \ref{prop:condhypo}}
For \eqref{cpresult0}, let us assume that (\cite[Theorem 1, Chapter 4]{protter2005stochastic}) 
\begin{equation*}
    X_t=M_t+V_t,
\end{equation*}
where $M_t$ is an $\FF$-local martingale and $V_t$ is a finite variation process. To prove that $X$ is an $\overline\FF$-semimartingale, it suffices to show that $M_t$ is an $\overline\FF$-local martingale. Let $\{\tau_k\}_{k\in\N^+}$ be a sequence of $\FF$-stopping times such that $M^{\tau_k}:=(M_{t\wedge\tau_k})_{t\in[0,T]}$ is an $\FF$-martingale for all $k\in\N^+$, and $\tau_k\to\infty$ as $k\to\infty$. We claim that
\begin{equation*}
    \overline\E[M_{t\wedge\tau_k}|\overline \F_s]=M_{s\wedge\tau_k}.
\end{equation*}
Recall that 
\begin{equation*}
    \overline\F_s=\sigma\{A_0\times A_1\times A_2:A_0,A_1,A_2\in\F_s\},
\end{equation*}
it suffices to show that
\begin{equation}\label{eqtmp11}   
    \overline\E\big[M_{t\wedge\tau_k}\mathbbm1_{A_0\times A_1\times A_2}\big]=\overline\E\big[M_{s\wedge\tau_k}\mathbbm1_{A_0\times A_1\times A_2}\big]
\end{equation}
for all $A_0,A_1,A_2\in\F_s$. Note that $M_{t\wedge\tau_k}\mathbbm1_{A_0\times \Omega\times \Omega}$, $\mathbbm1_{\Omega\times A_1\times\Omega}$ and  $\mathbbm1_{\Omega\times\Omega\times A_2}$  are conditionally independent given $\G_T$ by \eqref{conditionindep} in Proposition \ref{prop:condcopy}, we have
\begin{equation}\label{eqtmp8}
    \begin{aligned}
&\overline\E\big[M_{t\wedge\tau_k}\mathbbm1_{A_0\times A_1\times A_2}\big]=\overline\E\Big[\overline\E\big[M_{t\wedge\tau_k}\mathbbm1_{A_0\times \Omega\times \Omega}\big|\G_T\big]\overline\E\big[\mathbbm1_{\Omega\times A_1\times \Omega}\big|\G_T\big]\overline\E\big[\mathbbm1_{\Omega\times  \Omega\times A_2}\big|\G_T\big]\Big].
\end{aligned}
\end{equation}
By \eqref{conditionident} in Proposition \ref{prop:condcopy} and $\F_s\indep\G_T|\G_s$,   
\begin{equation}\label{eqtmp9}
    \overline\E\big[\mathbbm1_{\Omega\times A_1\times \Omega}\big|\G_T\big]=\overline\P(A_1|\G_T)=\overline\P(A_1|\G_s),\text{ and similarly }\overline\E\big[\mathbbm1_{\Omega\times \Omega\times A_2}\big|\G_T\big]=\overline\P(A_2|\G_s).
\end{equation}
Therefore \eqref{eqtmp8} becomes 
\begin{equation}\label{eqtmp10}
    \begin{aligned}
\overline\E\big[M_{t\wedge\tau_k}\mathbbm1_{A_0\times A_1\times A_2}\big]&=\overline\E\Big[\overline\E\big[M_{t\wedge\tau_k}\mathbbm1_{A_0\times \Omega\times \Omega}\big|\G_T\big]\overline\P\big( A_1|\G_s\big)\overline\P\big( A_2|\G_s\big)\Big]
\\
&=\overline\E\Big[\overline\E\big[M_{t\wedge\tau_k}\mathbbm1_{A_0\times\Omega\times\Omega}\big|\G_s\big]\overline\P\big( A_1|\G_s\big)\overline\P\big( A_2|\G_s\big)\Big],
\end{aligned}
\end{equation}
where the second equality holds by the tower rule. Since $M^{\tau_k}$ is an 
$\FF$-martingale, 
\begin{equation*}
    \overline\E\big[M_{t\wedge\tau_k}\mathbbm1_{A_0\times\Omega\times\Omega}\big|\G_s\big]=\overline\E\big[\mathbbm1_{A_0\times\Omega\times\Omega}\E[M_{t\wedge\tau_k}|\F_s]\big|\G_s\big]=\overline\E\big[M_{s\wedge\tau_k}\mathbbm1_{A_0\times\Omega\times\Omega}\big|\G_s\big].
\end{equation*}
Therefore, 
\begin{equation*}
    \begin{aligned}
    \overline\E\big[M_{t\wedge\tau_k}\mathbbm1_{A_0\times A_1\times A_2}\big]&=\overline\E\Big[\overline\E\big[M_{s\wedge\tau_k}\mathbbm1_{A_0\times\Omega\times\Omega}\big|\G_s\big]\overline\P\big( A_1|\G_s\big)\overline\P\big( A_2|\G_s\big)\Big].
    \end{aligned}
\end{equation*}
Repeating the calculation as in \eqref{eqtmp8}, \eqref{eqtmp9}, and \eqref{eqtmp10} yields
\begin{equation*}
    \begin{aligned}
    \overline\E\big[M_{s\wedge\tau_k}\mathbbm1_{A_0\times A_1\times A_2}\big]&=\overline\E\Big[\overline\E\big[M_{s\wedge\tau_k}\mathbbm1_{A_0\times\Omega\times\Omega}\big|\G_s\big]\overline\P\big( A_1|\G_s\big)\overline\P\big( A_2|\G_s\big)\Big],
    \end{aligned}
\end{equation*}
which leads to \eqref{eqtmp11}.

To show that $X'$ is an  $\overline\FF$-semimartingale, by the construction of the extended probability space, we can define $M',V'$, and $\tau_k'$ to be respectively the  conditionally independent copy of $M,V$, and $\tau_k$. Then it suffices to show that  
\begin{equation*}   
    \E\big[M'_{t\wedge\tau'_k}\mathbbm1_{A_0\times A_1\times A_2}\big]=\E\big[M'_{s\wedge\tau'_k}\mathbbm1_{A_0\times A_1\times A_2}\big].
\end{equation*}
Note that
\begin{equation*}
    \begin{aligned}
&\overline\E\big[M'_{t\wedge\tau_k'}\mathbbm1_{A_0\times A_1\times A_2}\big]=\overline\E\Big[\overline\E\big[M'_{t\wedge\tau'_k}\mathbbm1_{\Omega\times A_2\times \Omega}\big|\G_T\big]\overline\P(A_0|\G_s)\overline\P(A_2|\G_s)\Big]
\\
&=\overline\E\Big[\overline\E\big[M_{t\wedge\tau_k}\mathbbm1_{A_1\times \Omega\times \Omega}\big|\G_T\big]\overline\P(A_0|\G_s)\overline\P(A_2|\G_s)\Big]=\overline\E\Big[\overline\E\big[M_{t\wedge\tau_k}\mathbbm1_{A_1\times \Omega\times \Omega}\big|\G_s\big]\overline\P(A_0|\G_s)\overline\P(A_2|\G_s)\Big]
\\
&=\overline\E\Big[\overline\E\big[M_{s\wedge\tau_k}\mathbbm1_{A_1\times \Omega\times \Omega}\big|\G_s\big]\overline\P(A_0|\G_s)\overline\P(A_2|\G_s)\Big]=\overline\E\Big[\overline\E\big[M_{s\wedge\tau_k}\mathbbm1_{A_1\times \Omega\times \Omega}\big|\G_T\big]\overline\P(A_0|\G_s)\overline\P(A_2|\G_s)\Big]
\\
&=\overline\E\Big[\overline\E\big[M'_{s\wedge\tau'_k}\mathbbm1_{\Omega\times A_2\times \Omega}\big|\G_T\big]\overline\P(A_0|\G_s)\overline\P(A_2|\G_s)\Big]=\E\big[M'_{s\wedge\tau'_k}\mathbbm1_{A_0\times A_1\times A_2}\big],
\end{aligned}
\end{equation*}
where the first and last equalities mimic the argument in \eqref{eqtmp8}, \eqref{eqtmp9} and \eqref{eqtmp10}, the second and the sixth equalities hold from \eqref{conditionident} in Proposition \ref{prop:condcopy}, the third and the fifth equalities follow from the tower rule, and the fourth equality is the martingale property. This completes the proof of  \eqref{cpresult0}.

For \eqref{cpresult1}, it suffices to prove that 
\begin{equation*}
    \overline\E\bigg[\bigg(\sup_{\sigma_m}\sum_{i=1}^{k_m}\big|X^\G_{t_{i+1}^m}-X^\G_{t_{i}^m}\big|\bigg)^p\bigg]<\infty,
\end{equation*}
where the supreme is taken over all partition $\sigma_m:\{0=t_0^m\leq t_1^m\leq \cdots\leq t_{k_m}^m=T\}$  of $[0,T]$. Since
\begin{equation*}
    \sum_{i=1}^{k_m}\big|X^\G_{t_{i+1}^m}-X^\G_{t_{i}^m}\big|=\sum_{i=1}^{k_m}\Big|\overline\E\big[X_{t_{i+1}^m}-X_{t_{i}^m}\big|\G_T\big]\Big|\leq \overline\E\bigg[\sum_{i=1}^{k_m}\big|X_{t_{i+1}^m}-X_{t_{i}^m}\big|\Big|\G_T\bigg],
\end{equation*}
we have,
\begin{equation*}
    \begin{aligned}
    \overline\E&\bigg[\bigg(\sup_{\sigma_m}\sum_{i=1}^{k_m}\big|X^\G_{t_{i+1}^m}-X^\G_{t_{i}^m}\big|\bigg)^p\bigg]
    \\
    &=\overline\E\bigg[\sup_{\sigma_m}\bigg(\sum_{i=1}^{k_m}\big|X^\G_{t_{i+1}^m}-X^\G_{t_{i}^m}\big|\bigg)^p\bigg]
\leq\overline\E\bigg[\sup_{\sigma_m}\bigg(\overline\E\bigg[\sum_{i=1}^{k_m}\big|X_{t_{i+1}^m}-X_{t_{i}^m}\big|\Big|\G_T\bigg]\bigg)^p\bigg]
    \\
    &\leq\overline\E\bigg[\bigg(\E\bigg[\sup_{\sigma_m}\sum_{i=1}^{k_m}\big|X_{t_{i+1}^m}-X_{t_{i}^m}\big|\Big|\G_T\bigg]\bigg)^p\bigg]\leq\overline\E\bigg[\overline\E\bigg[\bigg(\sup_{\sigma_m}\sum_{i=1}^{k_m}\big|X_{t_{i+1}^m}-X_{t_{i}^m}\big|\bigg)^p\bigg|\G_T\bigg]\bigg]
    \\
    &= \overline\E\bigg[\bigg(\sup_{\sigma_m}\sum_{i=1}^{k_m}\big|X_{t_{i+1}^m}-X_{t_{i}^m}\big|\bigg)^p\bigg]=\overline\E\bigg[\bigg(\int_0^T\big|\d X_s\big|\bigg)^p\bigg]<\infty,
    \end{aligned}
\end{equation*}
which leads to \eqref{cpresult1}. 

For \eqref{cpresult2}, due to  \eqref{cpresult1} and the decomposability (see \cite[Theorem 1, Chapter 4]{protter2005stochastic}), it suffices to show that if 
$X$ is an $\FF$-local martingale, then $X^\G$ is a local martingale with respect to $\overline \FF$, $\FF$ and $\GG$, which thanks to \eqref{cpresult0}  can be further reduced to proving that $X^\G$ is a local martingale with respect to $\FF$ and $\GG$. Let $\tau_k$ be the stopping time $\tau_k:=\inf\{t\in[0,T]:[X,X]_t>k\}$ and consider the truncated process $X^{\tau_k}:=\{X_{t\wedge\tau_k}\}_{t\in[0,T]}$. By Burkholder-Davis-Gundy inequality, 
\begin{equation}\label{bddqv}
    \overline\E\big[(X^{\tau_k})^*_t\big]^2\leq\overline\E\big[\big((X^{\tau_k}\big)^*_t)^2\big]\leq C\overline\E\big[\big[X^{\tau_k},X^{\tau_k}\big]_t\big]\leq k,
\end{equation}
for some constant $C$ and $(X^{\tau_k})^*_t=\sup_{0\leq s\leq t}|X_{s\wedge\tau_k}|.$ Let $\{\widetilde \tau_\ell\}_{\ell\in\N^+}$ be the sequence of the $\FF$-stopping times such that $\tau_\ell\to T$ and $X^{\tilde\tau_\ell}=\{X_{t\wedge\tilde\tau_\ell}\}_{t\in[0,T]}$ is an $\FF$-martingale for any $\ell\in\N^+$. Then  $X^{\tilde\tau_\ell\wedge \tau_k}=\{X_{t\wedge\tilde\tau_\ell\wedge \tau_k}\}_{t\in[0,T]}$ is an $\FF$-martingale for any $\ell\in\N^+$.
 By Lemma \ref{lem:CDCT}, $|X_{t\wedge \tau_k\wedge\tilde\tau_\ell}|\leq (X^{\tau_k})^*_t$, and from \eqref{bddqv}, 
\begin{equation*}
    X_{s\wedge\tau_k}=\lim_{\ell\to\infty}X_{s\wedge\tau_k\wedge\tilde\tau_\ell}=\lim_{\ell\to\infty}\overline\E\big[X_{t\wedge\tau_k\wedge\tilde\tau_\ell}|\F_s\big]=\overline\E\Big[\lim_{\ell\to\infty}X_{t\wedge\tau_k\wedge\tilde\tau_\ell}|\F_s\Big]=\overline\E\big[X_{t\wedge\tau_k}|\F_s\big].
\end{equation*}
Combined with \eqref{bddqv}, we have that $X^{\tau_k}=\{X_{t\wedge\tau_k}\}_{t\in[0,T]}$ is indeed a square integrable $\FF$-martingale.

{Now we claim that $(X^{\tau_k})^\G$ is a square integrable $\GG$-martingale. (Hence, by the \textbf{H} hypothesis and Remark \ref{rmk:hypothesis} also an $\FF$-martingale).
The integrability holds trivially since $X_t$ is integrable, and the martingale property holds since }
\begin{equation*}
\begin{aligned}
\overline\E\Big[X_{t\wedge\tau_k}^\G\big|\G_s\Big]&=\overline\E\Big[\overline\E\big[X_{t\wedge\tau_k}|\G_T\big]\big|\G_s\Big]=\overline\E\big[X_{t\wedge\tau_k}|\G_s\big]=\overline\E\Big[\overline\E\big[X_{t\wedge\tau_k}|\F_s\big]\big|\G_s\Big]
\\
&=\overline\E\big[X_{s\wedge\tau_k}|\G_s\big]=\overline\E\big[X_{s\wedge\tau_k}|\G_T\big]=X_{s\wedge\tau_k}^\G.
\end{aligned}
\end{equation*}
This completes the proof that $X^\G$ is an $\GG$-local martingale and \eqref{cpresult2}.

\subsection{Proof of Theorem \ref{thm:condprocess}}\label{sec:pfprocond}


For \eqref{cpresult3}, let $\sigma_m:0=t_0^m\leq t_1^m\leq...\leq t_{k_m}^m=t$ be a sequence of partitions which tends to identity in the sense of \cite[Page 64, definition]{protter2005stochastic} and satisfies
\begin{equation*}   
     \int_0^tZ_{s-}\d X^\G_s= \lim_{m\to\infty}\sum_{k=0}^{k_m-1}Z_{t_k^m}(X^\G_{t_{k+1}^m}-X^\G_{t_{k}^m}),\;\; \int_0^tZ_{s-}\d X'_s= \lim_{m\to\infty}\sum_{k=0}^{k_m-1}Z_{t_k^m}(X'_{t_{k+1}^m}-X'_{t_{k}^m}),
\end{equation*}
where the convergence is understood as a.s. convergence.  Then
\begin{equation*}
\begin{aligned}
 \int_0^tZ_{s-}\d X^\G_s&= \lim_{m\to\infty}\sum_{k=0}^{k_m-1}Z_{t_k^m}(X^\G_{t_{k+1}^m}-X^\G_{t_{k}^m})=\lim_{m\to\infty}\sum_{k=0}^{k_m-1}Z_{t_k^m}\E\Big[X_{t_{k+1}^m}-X_{t_{k}^m}\Big|\G_T\Big]
 \\
 &=\lim_{m\to\infty}\sum_{k=0}^{k_m-1}Z_{t_k^m}\overline\E\Big[X'_{t_{k+1}^m}-X'_{t_{k}^m}\Big|\F\Big]=\lim_{m\to\infty}\sum_{k=0}^{k_m-1}\E\Big[Z_{t_k^m}(X'_{t_{k+1}^m}-X'_{t_{k}^m})\Big|\F\Big].
\end{aligned}
\end{equation*}
Define ${Z}^{m}$ as
\begin{equation*}
    Z^{m}_t=Z_0\mathbbm1_{\{0\}}(t)+\sum_{k=0}^{k_m-1}Z_{t_k^m}\mathbbm1_{(t_k^m,t_{k+1}^m]}(t),
\end{equation*}
then $\int_0^tZ_{s-}\d X^\G_s= \lim_{m\to\infty}\E\big[\int_0^tZ_s^m\d X'_s\Big|\F\big],$ and 
\begin{equation*}
    \sum_{k=0}^{k_m-1}\overline\E\Big[Z_{t_k^m}\big(X'_{t_{k+1}^m}-X'_{t_{k}^m}\big)\Big]-\overline\E\bigg[\int_0^tZ_{s-}\d X'_s\bigg]=\overline\E\bigg[\int_0^t(Z_s^m-Z_{s-})\d X'_s\bigg].
\end{equation*}
Recall that $X'$ has the decomposition $X'=M'+V'$, with $M'$  a local martingale, $V'$ an adapted finite variation process, and 
$\bigg\|\sqrt{[M',M']_T}+\int_0^T\big|\d V'_s\big|\bigg\|_{L^p}<\infty.$
Therefore,
\begin{equation}\label{equtemp3}
    \overline\E\bigg[\bigg|\int_0^t(Z_s^m-Z''_{s-})\d X'_s\bigg|\bigg]\leq\overline\E\bigg[\bigg|\int_0^t(Z_s^m-Z_{s-})\d M'_s\bigg|\bigg]+\overline\E\bigg[\bigg|\int_0^t(Z_s^m-Z_{s-})\d V'_s\bigg|\bigg].
\end{equation}
For the first term in \eqref{equtemp3}, applying Burkholder-Davis-Gundy inequality to get
\begin{equation*}
    \begin{aligned}
    \overline\E\bigg[\bigg|\int_0^t(Z_s^m-Z_{s-})\d M'_s\bigg|\bigg]&\leq \overline\E\bigg[\sup_{0\leq s\leq t}\bigg|\int_0^s(Z_u^m-Z_{u-})\d M'_u\bigg|^p\bigg]
    \\
    &\leq c_1\overline\E\bigg[\bigg(\int_0^t(Z_u^m-Z_{u-})^2\d [M',M']_s\bigg)^{p/2}\bigg]
    \end{aligned}
\end{equation*}
for some constant $c_1$. Note that $\int_0^t(Z_u^m-Z_{u-})^2\d [M',M']_s\leq 4\sup_{0\leq s\leq t}(Z_s)^2[M',M']_t$, and 
\begin{equation*}   
    \begin{aligned}
    &\overline\E\Big[\Big(\sup_{0\leq s\leq t}(Z_s)^2[M',M']_t\Big)^{1/2}\Big]\leq\bigg(\overline\E\Big[\sup_{0\leq s\leq t}(Z_s)^{q}\Big]\overline\E\big[[M',M']^{p/2}_t\big]\bigg)^{1/2}\leq C\sqrt{\|Z\|^q_{\H^q}\overline\E\big[[M',M']_t^{p/2}\big]}<\infty.
    \end{aligned}
\end{equation*}
Since $Z$ is  c\`adl\`ag,  $Z^m_s\to Z_{s-}, a.s.$ and $|Z^m_s-Z_s|\leq \sup_{0\leq u\leq s}|Z_u|\in S_q$; and Lemma \ref{lem:CDCT} leads to
 \begin{equation*}
    \overline\E\bigg[\int_0^t(Z_s^m-Z_{s-})\d M'_s\bigg]\to 0.
 \end{equation*}
 For the second term in \eqref{equtemp3}, we have
 \begin{equation*}
    \overline\E\bigg[\bigg|\int_0^t(Z_s^m-Z_{s-})\d V'_s\bigg|\bigg]\leq \overline\E\bigg[\int_0^t|Z_s^m-Z_{s-}|\big|\d V'_s\big|\bigg].
 \end{equation*}
 Note that $\int_0^t|Z_s^m-Z_{s-}|\big|\d V'_s\big|\leq 2\sup_{0\leq s\leq t}|Z_s|\int_0^t|\d V_t|$, and similarly, we have
 \begin{equation*}
    \overline\E\bigg[\bigg|\int_0^t(Z_s^m-Z_{s-})\d V'_s\bigg|\bigg]\to0,
 \end{equation*}
 and
 \begin{equation*}
    \lim_{m\to\infty}\overline\E\bigg[\bigg|\int_0^tZ_s^m\d X'_s-\int_0^tZ_{s-}\d X'_s\bigg|\bigg]=0.
 \end{equation*}
 Now define 
 \begin{equation*}
 R^m:=\overline\E\bigg[\int_0^tZ^m_{s-}\d X'_s\Big|\F\bigg], \;R:=\overline\E\bigg[\int_0^tZ_{s-}\d X'_s\Big|\F\bigg],
 \end{equation*}
 for all $A\in\F$, we have
 \begin{equation*}
 \begin{aligned}
    \overline\E[(R-R_m)\mathbbm1_A]=\overline\E\bigg[\bigg(\int_0^tZ_s^m\d X'_s-\int_0^tZ_{s-}\d X'_s\bigg)\mathbbm1_A\bigg]\leq \overline\E\bigg[\bigg|\int_0^tZ_s^m\d X'_s-\int_0^tZ_{s-}\d X'_s\bigg|\bigg]\to0.
    \end{aligned}
 \end{equation*}
 Let $A=\{R-R_m>\epsilon\}$, we have
 \begin{equation*}
    \overline\P\big(R_m-R>\epsilon\big)\leq \frac{\overline\E[(R_m-R)\mathbbm1_A]}{\epsilon}\to0,
 \end{equation*}
 and similarly  
 \begin{equation*}
    \overline\P\big(R_m-R<\epsilon\big)\to0.
 \end{equation*}
 Therefore, $R_m\to R$ in probability. Meanwhile, $R_m\to \int_0^tZ_{s-}\d X^\G_s$ a.s., hence 
 \begin{equation*}
     \int_0^tZ_{s-}\d X^\G_s=\overline\E\bigg[\int_0^tZ^m_{s-}\d X'_s\Big|\F\bigg].
 \end{equation*}

For \eqref{cpresult4}, let us first prove that 
\begin{equation*}
    [X^\G,Y^\G]_t=\overline\E\Big[[X'',Y']_t\Big|\G\Big].
\end{equation*}
By the definition of the quadratic variation,
\begin{equation*}
    \big[X^\G,Y^\G\big]_t=X^\G_tY^\G_t-\int_0^tX^\G_{s-}\d Y^\G_s-\int_0^tY^\G_{s-}\d X^\G_s,
\end{equation*}
where the last two terms are well defined because of \eqref{cpresult2}. The first term $X^\G_tY^\G_t$ is
\begin{equation*}
    X^\G_tY^\G_t=\overline\E\big[X_t|\G_T\big]\overline\E\big[Y_t|\G_T\big]=\overline\E\big[X''_t|\F\big]\overline\E\big[Y'_t|\F\big]=\overline\E\big[X''_tY'_t|\F\big],
\end{equation*}
where the last equation holds due to \eqref{conditionindep} in Proposition \ref{prop:condcopy}. For the second term $\int_0^tX^\G_{s-}\d Y^\G_s$, let $\sigma_m:0=t_0^m\leq t_1^m\leq...\leq t_{k_m}^m=t$ be a sequence of partitions  tending to  identity in the sense of \cite[Page 64, definition]{protter2005stochastic} and satisfying
\begin{equation*}   
     \int_0^tX^\G_{s-}\d Y^\G_s= \lim_{m\to\infty}\sum_{k=0}^{k_m-1}X^\G_{t_k^m}(Y^G_{t_{k+1}^m}-Y^\G_{t_{k}^m}),\;\; \int_0^tX''_{s-}\d Y'_s= \lim_{m\to\infty}\sum_{k=0}^{k_m-1}X''_{t_k^m}(Y'_{t_{k+1}^m}-Y'_{t_{k}^m}),
\end{equation*}
where the convergence is understood as a.s. convergence. Therefore,
\begin{equation*}  
    \begin{aligned}
  \int_0^tX^\G_{s-}\d Y^\G_s&= \lim_{m\to\infty}\sum_{k=0}^{k_m-1}X^\G_{t_k^m}(Y^\G_{t_{k+1}^m}-Y^\G_{t_{k}^m})=\lim_{m\to\infty}\sum_{k=0}^{k_m-1}\E\big[X_{t_k^m}|\G_T\big]\overline\E\big[Y_{t_{k+1}^m}-Y_{t_{k}^m}|\G_T\big]
  \\
  &=\lim_{m\to\infty}\sum_{k=0}^{k_m-1}\E\big[X''_{t_k^m}|\F\big]\E\big[Y'_{t_{k+1}^m}-Y'_{t_{k}^m}|\F\big]=\lim_{m\to\infty}\sum_{k=0}^{k_m-1}\overline\E\big[X''_{t_k^m}\big(Y'_{t_{k+1}^m}-Y'_{t_{k}^m}\big)|\F\big],
  \end{aligned}
\end{equation*}
where the last equation holds due to \eqref{conditionindep} in Proposition \ref{prop:condcopy}.
Similar to the proof of \eqref{cpresult3}, we  conclude that 
\begin{equation*}
    \int_0^tX^\G_{s-}\d Y^\G_s=\overline\E\bigg[\int_0^tX''_{s-}\d Y'_s\bigg|\F\bigg],
\end{equation*}
and similarly
\begin{equation*}
    \int_0^tY^\G_{s-}\d X^\G_s=\overline\E\bigg[\int_0^tY'_{s-}\d X''_s\bigg|\F\bigg].
\end{equation*}
Therefore
\begin{equation*}
    [X^\G,Y^\G]_t=\overline\E\Big[[X'',Y']_t\Big|\F\Big].
\end{equation*}
Meanwhile, we can further assume that the partition $\sigma_m:0=t_0^m\leq t_1^m\leq...\leq t_{k_m}^m=t$  satisfies
\begin{equation*}   
     \int_0^tZ_{s-}\d [X^\G,Y^\G]_s= \lim_{m\to\infty}\sum_{k=0}^{k_m-1}Z_{t_k^m}\big([X^\G,Y^\G]_{t_{k+1}^m}-[X^\G,Y^\G]_{t_{k}^m}\big),
\end{equation*}
and 
\begin{equation*}
\int_0^tZ_{s-}\d [X'',Y']_s= \lim_{m\to\infty}\sum_{k=0}^{k_m-1}Z_{t_k^m}([X'',Y']_{t_{k+1}^m}-[X'',Y']_{t_{k}^m}).
\end{equation*}
Repeat the same argument as before and notice that $$\big|\int_0^t(Z_s^m-Z_{s-})\d[X'',Y']_s\big|\leq 2\sup_{0\leq s\leq t}|Z_s|[X'',Y']_t. $$ Then we have 
\begin{equation*}
    \begin{aligned}
\overline\E\bigg[\sup_{0\leq s\leq t}|Z_s|[X'',Y']_s\bigg]&\leq \overline\E[(\sup_{0\leq s\leq t}|Z_s|)^r]^{1/r}\overline\E\Big[[X'',X'']_t^p\Big]^{1/p}\overline\E\Big[[Y',Y']_t^q\Big]^{1/q}
\\
&\leq C\|Z\|_{H_r}\|X\|_{H_p}\|Y\|_{H_q},
\end{aligned}
\end{equation*}
where the first inequality is by Cauchy inequality and Kunita-Watanabe inequality with $C$  a generic constant. Now repeating the arguments in \eqref{cpresult3}, we achieve
\begin{equation*}
    \int_0^t Z_{s-}\d \big[X^\G,Y^\G\big]_t=\overline\E\bigg[\int_0^tZ_{s-}\d\big[X'',Y'\big]_s\bigg|\F\bigg],
\end{equation*}
and similarly, 
\begin{equation*}
    \int_0^t Z_{s-}\big[X^\G,Y\big]_s=\overline\E\bigg[\int_0^t Z_{s-}\big[X',Y\big]_s\bigg|\F\bigg],
\end{equation*}
which completes the  proof.\qed

{  
\section{Application to control problems with partial observed jump-diffusions}\label{sec:appl}
We now apply  our result (Theorem \ref{mainthm}) to analyze control problems with partially observed jump-diffusions. We will derive an appropriate form of the verification theorem and recover the results in \cite{bandini2019randomized} as  special cases. 

Consider a complete filtered space $(\Omega,\F,\FF=(\F_t)_{t\in[0,T]},\P)$ with a sub-filtration $\GG=(\G_t)_{t\in[0,T]}\subset\FF$ satisfying the compatibility condition as in Assumption \ref{asp:filt}. We define as an admissible control process any $\GG$-progressive process $\alpha$ with values in some given Borel space $A$. The set of admissible control processes is denoted by $\A^\GG$. 

Given a square-integrable initial condition $\bm \xi\in\R^n$, $d$-dimensional $\FF$-Brownian motion $\bm B$, an independent Poisson random measure $N$ with a L\'evy measure $\nu$, with $\widetilde N(\ d\bm\theta,\d t):=N(\d\bm\theta,\d t)-\nu(\d\bm\theta)\d t$ the compensated Poisson random measure and $\alpha \in \A^\GG$, the $n$-dimensional controlled state equation is defined by
\begin{align}\label{controlleddiffusion}
\d \bm X_t^{\alpha}=\bm b\big(\bm X_t^{\alpha},\alpha_t\big)\d t+\bm \sigma\big(\bm X_{t}^{\alpha},\alpha_t\big)\d \bm B_t+\int_{\R^q}\bm \beta(\bm X^{\alpha}_{t-},\alpha_{t-},\bm\theta)\widetilde N(\d\bm\theta,\d t),\text{ for }0\leq t\leq T,\;\bm X^\alpha_0=\bm \xi.
\end{align}
The coefficients $\bm b$ and $\bm \sigma$ are deterministic measurable functions from $\R^n \times A$ into $\R^n$ and $\R^{n\times d}$, respectively, and the coefficient  $\bm\beta$ are deterministic measurable functions from $\R^n \times A\times\R^q$ into $\R^n$ and $\R^{n\times d}$  assumed to satisfy the following standing assumptions: there exists some positive constant $C_1$ such that for all $\bm x,\bm x'\in\R^n$, $a\in A$,
\begin{align*}
    \big|\bm b(\bm x,a)-\bm b(\bm x',a)\big|+\big|\bm \sigma(\bm x,a)-\bm \sigma(\bm x',a)\big|+\int_{\R^q}\big|\bm\beta(\bm x,a,\bm\theta)-\bm\beta(\bm x',a,\bm\theta)\big|\nu(\d\bm\theta)&\leq C_1|\bm x-\bm x'|;
    \\
    \big|\bm b(\bm 0,a)\big|+\big|\bm \sigma(\bm 0,a)\big|+\int_{\R^q}\big|\bm\beta(\bm 0,a,\bm\theta)\big|\nu(\d\bm\theta)&\leq C_1.
\end{align*}
It is clear from  standard arguments that there exists a unique strong solution $\{\bm X_t^{\alpha}\}_{t\in[0,T]}$ to \eqref{controlleddiffusion}, which is $\FF$-adapted and satisfies 
$$
\sup_{\alpha\in\A^\GG}\E\bigg[\sup_{0\leq t\leq T} \big|\bm X_t^\alpha\big| ^2 \bigg]\leq  C\Big(1+¯\E\big[|\bm\xi|^2\big]\Big) < \infty,$$
for some positive constant $C$. 

The goal is to maximize, over all admissible control processes $\alpha \in \A^\GG$, the following objective:
\begin{align*}
    \J(\alpha):=\E\bigg[\int_0^Tf(\bm X_t^\alpha,\alpha_t)\d t+g(\bm X_T^\alpha)\bigg].
\end{align*}
Here $f:\R^n \times A \to \R$, and $g:\R^n \to \R$ are continuous functions satisfying the quadratic growth condition:
there exists some positive constant $C_2$ such that for all $(\bm x,a) \in \R^n \times A$, 
$$\big|f(\bm x,a)\big| + \big|g(\bm x)\big| \leq C_2 \big(1+|\bm x|^2\big) .$$
We can now define the objective function starting from time $t$, and for $\mu\in \PP_2(\R^d)$ by
$$
J(t,\alpha,\mu):=\E\bigg[\int_t^Tf(\bm X_s^\alpha,\alpha_s)\d s+g(\bm X_T^\alpha)\Big|\text{Law}(\bm X^\alpha_t|\G_T)=\text{Law}(\bm X^\alpha_t|\G_t)=\mu\bigg],
$$
and the corresponding value function by 
$$
V(t,\mu):=\sup_{\alpha}J(t,\alpha,\mu).
$$
It is not hard to see that both $J$ and $V$ satisfy the quadratic growth estimate:
$$
|J(t,\alpha,\mu)|\leq |V(t,\mu)|\leq C\Big(1+¯\E\big[|\bm\xi|^2\big]\Big)
$$
for all $\mu\in\PP_2(\R^d),\alpha\in\A^\G$ for some constant $C$. 

Note that \cite{bandini2019randomized} considers the case where the information $\GG$ is a filtration generated by a lower-dimensional Brownian motion $\bm W$ such that $\bm B=(\bm V,\bm W)$ and $\bm V$ independent of $\bm W$. This corresponds to the special case of \eqref{controlleddiffusion} with $\bm\beta =0$. In this case,  a form of It\^o's formula without proof is presented 
(see (5.7) of \cite{bandini2019randomized}), from which a verification theorem for the control problem with partial observation is derived.  Here we will derive rigorously,   as a corollary of Theorem \ref{mainthm}, a generalization of 
  (5.7) in 
\cite{bandini2019randomized} with a general form of $\GG$ satisfying the compatibility condition.

\begin{corollary}
[It\^o's formula for flows of partial observed jump-diffusion]
   Consider a complete filtered space $(\Omega,\F,\FF=(\F_t)_{t\in[0,T]},\P)$ with a sub-filtration $\GG=(\G_t)_{t\in[0,T]}$ satisfying the compatibility condition in Assumption \ref{asp:filt}. Let $\alpha\in\A^\GG$ be an admissible control and  the  $n$-dimensional controlled state $\bm X_t^\alpha$ satisfying \eqref{controlleddiffusion}, and set $\mu_t=$Law$(\bm X^\alpha_t|\G_t)$. Then, for $\Phi\in C^{2}(\mathcal{P}_2(\R^n))$, we have
   \begin{equation}
   \begin{aligned}
   &\Phi(\mu_t)-\Phi(\mu_0)
   \\
   =&\;\overline\E\Bigg[\int_{0}^t\bigg(\Big(\nabla_{\bm x_1}\frac{\delta \Phi}{\delta\mu}\big(\mu_{s},\bm X^{\alpha,1}_{s}\big)\cdot\bm b(\bm X^{\alpha,1}_s,\alpha_s)+\frac12\nabla^2_{\bm x_1\bm x_1}\frac{\delta \Phi}{\delta\mu}\big(\mu_{s},\bm X^{\alpha,1}_{s}\big):\bm \sigma\bm \sigma^\top\big(\bm X^{\alpha,1}_s,\alpha_s\big)\Big)\d s
   \\
   &\qquad\qquad+\frac12\nabla^2_{\bm x_1\bm x_2}\frac{\delta^2 \Phi}{(\delta\mu)^2}\big(\mu_{s},\bm X^{\alpha,1}_{s},\bm X^{\alpha,2}_{s}\big):\bm \sigma\big(\bm X^{\alpha,1}_s,\alpha_s\big)\bm \sigma^\top\big(\bm X^{\alpha,2}_s,\alpha_s\big)\d [\bm B^1,\bm B^2]_s
   \\
   &\qquad\qquad+\nabla_{\bm x_1}\frac{\delta \Phi}{\delta\mu}\big(\mu_{s},\bm X^{\alpha,1}_{s}\big)\cdot\bm \sigma(\bm X^{\alpha,1}_s,\alpha_s)\d \bm B^1_s
   \bigg)
   \\
   &+\sum_{0<s\leq t}\bigg(\frac12\bigg(\frac{\delta^2 \Phi}{(\delta\mu)^2}\big(\mu_{s-},\bm X^{\alpha,1}_{s},\bm X^{\alpha,2}_{s}\big)-\frac{\delta^2 \Phi}{(\delta\mu)^2}\big(\mu_{s-},\bm X^{\alpha,1}_{s-},\bm X^{\alpha,2}_{s}\big)
    \\
    &\qquad\qquad\qquad-\frac{\delta^2 \Phi}{(\delta\mu)^2}\big(\mu_{s-},\bm X^{\alpha,1}_{s},\bm X^{\alpha,2}_{s-}\big)+\frac{\delta^2 \Phi}{(\delta\mu)^2}\big(\mu_{s-},\bm X^{\alpha,1}_{s-},\bm X^{\alpha,2}_{s-}\big)\bigg)
    \\
    &\quad\qquad\qquad+\frac{\delta \Phi}{\delta\mu}\big(\mu_{s-},\bm X^{\alpha,1}_s\big)-\frac{\delta \Phi}{\delta\mu}\big(\mu_{s-},\bm X^{\alpha,1}_{s-}\big)\bigg)\mathbbm {1}_{\{\mu_s=\mu_{s-}\}}
    \bigg|\F\bigg]
    \\
   &+\sum_{0<s\leq t}\Big(\Phi(\mu_s)-\Phi(\mu_{s-})\Big),
   \end{aligned}
   \end{equation}
   for all $t\in[0,T]$, where $(\bm X^{\alpha,1},\bm B^1,N^1),(\bm X^{\alpha,1},\bm B^2,N^2)$ satisfying for $0\leq t\leq T$,
   \begin{align*}
   \d \bm X_t^{\alpha,1}=\bm b\big(\bm X_t^{\alpha,1},\alpha_t\big)\d t+\bm \sigma\big(\bm X_t^{\alpha,1},\alpha_t\big)\d \bm B^1_t+\int_{\R^q}\bm \beta(\bm X^{\alpha,1}_{t-},\alpha_{t-},\bm\theta)N^1(\d\bm\theta,\d t),
   \\
   \d \bm X_t^{\alpha,2}=\bm b\big(\bm X_t^{\alpha,2},\alpha_t\big)\d t+\bm \sigma\big(\bm X_t^{\alpha,2},\alpha_t\big)\d \bm B^2_t+\int_{\R^q}\bm \beta(\bm X^{\alpha,2}_{t-},\alpha_{t-},\bm\theta)N^2(\d\bm\theta,\d t),
   \end{align*}
   are two conditionally independent copies of $(\bm X^{\alpha},\bm B,N)$ given the sub $\sigma$-algebra $\G_T$ defined in an enlarged probability space $(\overline\Omega,\overline{\F},\overline{\P})$, and $\overline{\E}[\cdot|\F]$ is the conditional expectation given $\F$ in the enlarged probability space.    
\end{corollary}
\noindent\textbf{HJB equations and verification theorem.}  In the case where $\GG$ is the filtration generated by Brownian motion $\bm W$, with $\bm B=(\bm V,\bm W)$ and $\bm V \indep\bm W$, we can rewrite the dynamic by
$$
\d \bm X_t^{\alpha}=\bm b\big(\bm X_t^{\alpha},\alpha_t\big)\d t+\bm \sigma^{\bm V}\big(\bm X_{t}^{\alpha},\alpha_t\big)\d \bm V_t+\bm \sigma^{\bm W}\big(\bm X_{t}^{\alpha},\alpha_t\big)\d \bm W_t+\int_{\R^q}\bm \beta(\bm X^\alpha_{t-},\alpha_{t-},\bm\theta)\widetilde N(\d\bm\theta,\d t),$$
and the conditionally independent copies
\begin{align*}
\d \bm X_t^{\alpha,1}=\bm b\big(\bm X_t^{\alpha,1},\alpha_t\big)\d t+\bm \sigma^{\bm V}\big(\bm X_{t}^{\alpha,1},\alpha_t\big)\d \bm V^1_t+\bm \sigma^{\bm W}\big(\bm X_{t}^{\alpha,1},\alpha_t\big)\d \bm W_t+\int_{\R^q}\bm \beta(\bm X^{\alpha,1}_{t-},\alpha_{t-},\bm\theta)\widetilde N^1(\d\bm\theta,\d t),
\\
\d \bm X_t^{\alpha,2}=\bm b\big(\bm X_t^{\alpha,2},\alpha_t\big)\d t+\bm \sigma^{\bm V}\big(\bm X_{t}^{\alpha,2},\alpha_t\big)\d \bm V^2_t+\bm \sigma^{\bm W}\big(\bm X_{t}^{\alpha,2},\alpha_t\big)\d \bm W_t+\int_{\R^q}\bm \beta(\bm X^{\alpha,2}_{t-},\alpha_{t-},\bm\theta)\widetilde N^2(\d\bm\theta,\d t).
\end{align*}
Note that under this setting, at time $s$, $\bm X^{\alpha,1}_{s-}$ and  $\bm X^{\alpha,2}_{s-}$ has a jump $\bm \beta(\bm X^{\alpha,1}_{s-},\alpha_{s-},\theta^1),\bm \beta(\bm X^{\alpha,2}_{s-},\alpha_{s-},\theta^2)$ with Poisson random measure $N^1(\d \theta^1,\d s)\times N^2(\d \theta^2,\d s)$, we can simplify the It\^o's formula by
    \begin{equation}\label{ito_jumppartial}
    \begin{aligned}
    &\Phi(\mu_t)-\Phi(\mu_0)
    \\
    =&\;\int_{0}^t\overline\E\Big[\nabla_{\bm x_1}\frac{\delta \Phi}{\delta\mu}\big(\mu_{s},\bm X^{\alpha,1}_{s}\big)\cdot\bm b(\bm X^{\alpha,1}_s,\alpha_s)+\frac12\nabla^2_{\bm x_1\bm x_1}\frac{\delta \Phi}{\delta\mu}\big(\mu_{s},\bm X^{\alpha,1}_{s}\big):\bm \sigma\bm \sigma^\top\big(\bm X^{\alpha,1}_s,\alpha_s\big)
    \\
    &\qquad\qquad+\frac12\nabla^2_{\bm x_1\bm x_2}\frac{\delta^2 \Phi}{(\delta\mu)^2}\big(\mu_{s},\bm X^{\alpha,1}_{s},\bm X^{\alpha,2}_{s}\big):\bm \sigma^{\bm W}\big(\bm X^{\alpha,1}_s,\alpha_s\big){\bm \sigma^{\bm W}}^\top\big(\bm X^{\alpha,2}_s,\alpha_s\big)\Big|\F\Big]\d s
    \\
    &\quad +\frac12\int_0^t\int_{\R^{2q}}\overline\E\Big[\frac{\delta^2 \Phi}{(\delta\mu)^2}\big(\mu_{s},\bm X^{\alpha,1}_{s}+\bm \beta(\bm X^{\alpha,1}_{s},\alpha_{s},\bm \theta^1),\bm X^{\alpha,2}_{s}+\bm \beta(\bm X^{\alpha,2}_{s},\alpha_{s},\bm \theta^2)\big)
    \\
    &-\frac{\delta^2 \Phi}{(\delta\mu)^2}\big(\mu_{s},\bm X^{\alpha,1}_{s},\bm X^{\alpha,2}_{s}+\bm \beta(\bm X^{\alpha,2}_{s},\alpha_{s},\bm \theta^2)\big)-\frac{\delta^2 \Phi}{(\delta\mu)^2}\big(\mu_{s},\bm X^{\alpha,1}_{s}+\bm \beta(\bm X^{\alpha,1}_{s},\alpha_{s},\bm \theta^1),\bm X^{\alpha,2}_{s}\big)
    \\
    &\qquad\qquad\qquad\qquad+\frac{\delta^2 \Phi}{(\delta\mu)^2}\big(\mu_{s},\bm X^{\alpha,1}_{s},\bm X^{\alpha,2}_{s}\big)\Big|\F\Big]\nu(\d\bm\theta^1)\nu(\d\bm\theta^2)\d s
    \\
    &\quad\qquad+\int_0^t\int_{\R^{q}}\overline\E\Big[\frac{\delta \Phi}{\delta\mu}\big(\mu_{s-},\bm X^{\alpha,1}_s+\bm \beta(\bm X^{\alpha,1}_{s},\alpha_{s},\bm \theta)\big)-\frac{\delta \Phi}{\delta\mu}\big(\mu_{s-},\bm X^{\alpha,1}_{s}\big)
    \\
    &\qquad\qquad\qquad\qquad\quad-\nabla_{\bm x_1}\frac{\delta \Phi}{\delta\mu}\big(\mu_{s-},\bm X^{\alpha,1}_{s}\big)\cdot \bm \beta(\bm X^{\alpha,1}_{s},\alpha_{s},\bm \theta)\Big|\F\Big]\nu(\d \bm\theta)\d s
    \\
    &+\int_0^t\overline \E\Big[\nabla_{\bm x_1}\frac{\delta \Phi}{\delta\mu}\big(\mu_{s},\bm X^{\alpha,1}_{s}\big)\cdot\bm \sigma^{\bm W}(\bm X^{\alpha,1}_s,\alpha_s)\Big|\F\Big]\d \bm W_s,
    \end{aligned}
    \end{equation}
Proceeding by the same argument as in \cite{cosso2019zero}, we have the dynamic programming principle  for the control problem with partial observation. In particular, for all $0\leq t<s\leq T$,
$$
V(t,\mu)=\inf_{\alpha\in \A^G}\E\bigg[\int_t^sf(\bm X^\alpha_u,\alpha_u)\d u+V\big(s,\text{Law}(\bm X_s^\alpha|\G_T)\big)\bigg|\text{Law}(\bm X_t^\alpha|\G_T)=\mu\bigg].
$$
We have now the ingredients for deriving the Bellman equation. For $\Phi \in C^{2,2}(\PP_2 (\R^n))$, $\mu\in\PP_2 (\R^n)$ and $a\in A$, the operator $\mathcal{L}^{\mu,a}$ maps from $C^{2}(\PP_2 (\R^n))$ to a function $\R^n\to\R$ defined by
\begin{align*}
\big(\mathcal{L}^{\mu,a}\Phi\big) (\bm x):= &\;\nabla_{\bm x_1}\frac{\delta \Phi}{\delta\mu}\big(\mu,\bm x\big)\cdot \bm b(\bm x,a) +\frac12\nabla^2_{\bm x_1\bm x_1}\frac{\delta \Phi}{\delta\mu}\big(\mu,\bm x\big):\bm \sigma\bm \sigma^\top\big(\bm x,a\big)
\\
&+\int_{\R^{q}}\bigg(\frac{\delta \Phi}{\delta\mu}\big(\mu,\bm x+\bm \beta(\bm x,a,\bm \theta)\big)-\frac{\delta \Phi}{\delta\mu}\big(\mu,\bm x\big)-\nabla_{\bm x_1}\frac{\delta \Phi}{\delta\mu}\big(\mu,\bm x\big)\cdot \bm \beta(\bm x,a,\bm \theta)\bigg)\nu(\d \bm\theta).
\end{align*} 
Furthermore, for $\Phi \in C^{2,2}(\PP_2 (\R^n))$, $\mu\in\PP_2 (\R^n)$, and $a\in A$, the operator $\mathcal{M}^{\mu,a}$ maps from $C^{2,2}(\PP_2 (\R^n))$ to a function $\R^{2n}\to\R$ defined by
\begin{align*}
\big(\mathcal{M}^{\mu,a}\Phi\big) (\bm x,\bm x'):= &\;\frac12\nabla^2_{\bm x_1\bm x_2}\frac{\delta^2 \Phi}{(\delta\mu)^2}\big(\mu,\bm x,\bm x'\big):\bm \sigma^{\bm W}\big(\bm x,a\big){\bm \sigma^{\bm W}}^\top\big(\bm x',a\big)\Big)
\\
&+\int_{\R^{2q}}\bigg(\frac{\delta^2 \Phi}{(\delta\mu)^2}\big(\mu,\bm x+\bm \beta(\bm x,a,\bm \theta^1),\bm x'+\bm \beta(\bm x',a,\bm\theta^2)\big)
\\
&-\frac{\delta^2 \Phi}{(\delta\mu)^2}\big(\mu,\bm x,\bm x'+\bm \beta(\bm x',a,\bm\theta^2)\big)-\frac{\delta^2 \Phi}{(\delta\mu)^2}\big(\mu,\bm x+\bm \beta(\bm x,a,\bm \theta^1),\bm x'\big)
\\
&\qquad\qquad+\frac{\delta^2 \Phi}{(\delta\mu)^2}\big(\mu,\bm x,\bm x'\big)\bigg)\nu(\d \bm\theta^1)\nu(\d \bm\theta^2).
\end{align*}
Now, the Bellman equation takes the following form:
\begin{equation}\label{hjb_jump}
    \begin{cases}
        \partial_t v(t,\mu)+\sup_{a\in A}\Big(\big\langle \mu,f(\cdot,a)+\mathcal{L}^{\mu,a}v(t,\mu)(\cdot)\big\rangle+\big\langle\mu\otimes\mu,\mathcal{M}^{\mu,a}v(t,\mu)(\cdot,\cdot)\big\rangle\Big)=0,
        \\
        v(T,\mu)= \big\langle \mu,g\big\rangle.
    \end{cases}
\end{equation}
The Bellman equation \eqref{hjb_jump} is validated by the following verification theorem, which can be proved by  following the similar argument as in \cite{bandini2019randomized}, with  the  It\^o's formula \eqref{ito_jumppartial}.

\begin{theorem}\label{thm:verification}
    Let $v$ be a real-valued function in $C^{1,2}([0,T]\times \PP_2 (\R^n))$, i.e., $v$ is continuous on $[0,T]\times \PP_2 (\R^n)$, $v(t,\cdot) \in C^2(\PP_2 (\R^n))$ for all $t \in [0,T]$, and $v(\cdot,\mu) \in C^1([0,T))$ for all $\mu \in \PP_2 (\R^n)$. Suppose that $v$ is a solution to the Bellman equation \eqref{hjb_jump}, and there exists for all $(t,\mu) \in [0,T) \times \PP_2 (\R^n)$ an element $\hat a(t,\mu)$ attaining the supremum in \eqref{hjb_jump} such that for a given  square integrable initial condition, the stochastic McKean-Vlasov SDE 
    \begin{align*}
    \d \hat{\bm X}_t=&\bm b\Big(\hat{\bm X}_t,\hat a\big(t,\text{Law}(\hat {\bm X_t}|\bm W)\big)\Big)\d t+\bm \sigma^{\bm V}\Big(\hat{\bm X}_t,\hat a\big(t,\text{Law}(\hat {\bm X_t}|\bm W)\big)\Big)\d \bm V_t
    \\
    &+\bm \sigma^{\bm W}\Big(\hat{\bm X}_t,\hat a\big(t,\text{Law}(\hat {\bm X_t}|\bm W)\big)\Big)\d \bm W_t+\int_{\R^q}\bm \beta\Big(\hat{\bm X}_t,\hat a\big(t,\text{Law}(\hat {\bm X_t}|\bm W)\big),\bm\theta\Big)N(\d\bm\theta,\d t)
    \end{align*}
    admits a unique solution $\hat{\bm X}$. Then, $v = V$ and the feedback control in $\A^\G$ defined by 
    $$\alpha^\star_s := \hat a\big(s,\text{Law}(\hat {\bm X_s}|\bm W)\big),\text{ for }t\leq s\leq T,$$
    is an optimal control for $V(t,\mu)$. That is, $V(t,\mu)=J(t,\alpha^\star,\mu)$.
\end{theorem}
Note that this generalizes the verification results in \cite{bandini2019randomized} to the case of jump-diffusions.
}

\bibliographystyle{plain}
\bibliography{references}

\begin{thebibliography}{10}

\bibitem{Ah}
Saran Ahuja.
\newblock Wellposedness of mean field games with common noise under a weak
  monotonicity condition.
\newblock {\em SIAM J. Control Optim.}, 54(1):30--48, 2016.

\bibitem{ARY}
Saran Ahuja, Weiluo Ren, and Tzu-Wei Yang.
\newblock Forward-backward stochastic differential equations with monotone
  functionals and mean field games with common noise.
\newblock {\em Stochastic Process. Appl.}, 129(10):3859--3892, 2019.

\bibitem{aksamit2017enlargement}
Anna Aksamit and Monique Jeanblanc.
\newblock {\em Enlargement of filtration with finance in view}.
\newblock Springer, 2017.

\bibitem{bandini2019randomized}
Elena Bandini, Andrea Cosso, Marco Fuhrman, and Huy{\^e}n Pham.
\newblock Randomized filtering and {B}ellman equation in {W}asserstein space
  for partial observation control problem.
\newblock {\em Stochastic Processes and their Applications}, 129(2):674--711,
  2019.

\bibitem{BZ}
Erhan Bayraktar and Yuchong Zhang.
\newblock A rank-based mean field game in the strong formulation.
\newblock {\em Electron. Commun. Probab.}, 21:Paper No. 72, 12, 2016.

\bibitem{beiglbock2018denseness}
Mathias Beiglb{\"o}ck and Daniel Lacker.
\newblock Denseness of adapted processes among causal couplings.
\newblock {\em arXiv preprint arXiv:1805.03185}, 2018.

\bibitem{BFY}
Alain Bensoussan, Jens Frehse, and Sheung Chi~Phillip Yam.
\newblock The master equation in mean field theory.
\newblock {\em J. Math. Pures Appl. (9)}, 103(6):1441--1474, 2015.

\bibitem{bremaud1978changes}
Pierre Br{\'e}maud and Marc Yor.
\newblock Changes of filtrations and of probability measures.
\newblock {\em Zeitschrift f{\"u}r Wahrscheinlichkeitstheorie und verwandte
  Gebiete}, 45(4):269--295, 1978.

\bibitem{buckdahn2017mean}
Rainer Buckdahn, Juan Li, Shige Peng, and Catherine Rainer.
\newblock Mean-field stochastic differential equations and associated {PDE}s.
\newblock {\em Annals of probability}, 45(2):824--878, 2017.

\bibitem{burzoni2020viscosity}
Matteo Burzoni, Vincenzo Ignazio, A~Max Reppen, and H~Mete Soner.
\newblock Viscosity solutions for controlled {M}ckean--{V}lasov
  jump-diffusions.
\newblock {\em SIAM Journal on Control and Optimization}, 58(3):1676--1699,
  2020.

\bibitem{cardaliaguet2019master}
Pierre Cardaliaguet, Fran{\c{c}}ois Delarue, Jean-Michel Lasry, and
  Pierre-Louis Lions.
\newblock {\em The master equation and the convergence problem in mean field
  games:(AMS-201)}.
\newblock Princeton University Press, 2019.

\bibitem{CDL}
Ren\'{e} Carmona, Fran\c{c}ois Delarue, and Daniel Lacker.
\newblock Mean field games with common noise.
\newblock {\em Ann. Probab.}, 44(6):3740--3803, 2016.

\bibitem{carmona2014master}
Ren{\'e} Carmona and Fran{\c{c}}ois Delarue.
\newblock The master equation for large population equilibriums.
\newblock In {\em Stochastic Analysis and Applications 2014: In Honour of Terry
  Lyons}, pages 77--128. Springer, 2014.

\bibitem{carmona2018probabilistic}
Ren{\'e} Carmona and Fran{\c{c}}ois Delarue.
\newblock {\em Probabilistic theory of mean field games with applications
  I-II}.
\newblock Springer, 2018.

\bibitem{carmona2016mean}
Ren{\'e} Carmona, Fran{\c{c}}ois Delarue, and Daniel Lacker.
\newblock Mean field games with common noise.
\newblock {\em Annals of Probability}, 44(6):3740--3803, 2016.

\bibitem{carmona2017mean}
Ren{\'e} Carmona, Fran{\c{c}}ois Delarue, and Daniel Lacker.
\newblock Mean field games of timing and models for bank runs.
\newblock {\em Applied Mathematics \& Optimization}, 76:217--260, 2017.

\bibitem{CFS}
Ren\'{e} Carmona, Jean-Pierre Fouque, and Li-Hsien Sun.
\newblock Mean field games and systemic risk.
\newblock {\em Commun. Math. Sci.}, 13(4):911--933, 2015.

\bibitem{chassagneux2022probabilistic}
Jean-Fran{\c{c}}ois Chassagneux, Dan Crisan, and Fran{\c{c}}ois Delarue.
\newblock {\em A probabilistic approach to classical solutions of the master
  equation for large population equilibria}, volume 280.
\newblock American Mathematical Society, 2022.

\bibitem{CF}
Michele Coghi and Franco Flandoli.
\newblock Propagation of chaos for interacting particles subject to
  environmental noise.
\newblock {\em Ann. Appl. Probab.}, 26(3):1407--1442, 2016.

\bibitem{cosso2019zero}
Andrea Cosso and Huy{\^e}n Pham.
\newblock Zero-sum stochastic differential games of generalized
  {M}ckean--{V}lasov type.
\newblock {\em Journal de Math{\'e}matiques Pures et Appliqu{\'e}es},
  129:180--212, 2019.

\bibitem{cox2024controlled}
Alexander~MG Cox, Sigrid K{\"a}llblad, Martin Larsson, and Sara Svaluto-Ferro.
\newblock Controlled measure-valued martingales: a viscosity solution approach.
\newblock {\em The Annals of Applied Probability}, 34(2):1987--2035, 2024.

\bibitem{cuchiero2025functional}
Christa Cuchiero, Xin Guo, and Francesca Primavera.
\newblock Functional {I}t$\backslash$\^{} o-formula and {T}aylor expansions for
  non-anticipative maps of c$\backslash$adl$\backslash$ag rough paths.
\newblock {\em arXiv preprint arXiv:2504.06164}, 2025.

\bibitem{cuchiero2019probability}
Christa Cuchiero, Martin Larsson, and Sara Svaluto-Ferro.
\newblock Probability measure-valued polynomial diffusions.
\newblock 2019.

\bibitem{DV}
Donald Dawson and Jean Vaillancourt.
\newblock Stochastic {M}c{K}ean-{V}lasov equations.
\newblock {\em NoDEA Nonlinear Differential Equations Appl.}, 2(2):199--229,
  1995.

\bibitem{dos2023ito}
Gon{\c{c}}alo Dos~Reis and Vadim Platonov.
\newblock It{\^o}-{W}entzell-{L}ions formula for measure dependent random
  fields under full and conditional measure flows.
\newblock {\em Potential Analysis}, 59(3):1313--1344, 2023.

\bibitem{elliott2000models}
Robert~J Elliott, Monique Jeanblanc, and Marc Yor.
\newblock On models of default risk.
\newblock {\em Mathematical Finance}, 10(2):179--195, 2000.

\bibitem{fleming1979some}
Wendell~H Fleming and Michel Viot.
\newblock Some measure-valued markov processes in population genetics theory.
\newblock {\em Indiana University Mathematics Journal}, 28(5):817--843, 1979.

\bibitem{fu2017mean}
Guanxing Fu and Ulrich Horst.
\newblock Mean field games with singular controls.
\newblock {\em SIAM Journal on Control and Optimization}, 55(6):3833--3868,
  2017.

\bibitem{grigorian2023enlargement}
Karen Grigorian and Robert~A Jarrow.
\newblock Enlargement of filtrations: An exposition of core ideas with
  financial examples.
\newblock {\em arXiv preprint arXiv:2303.03573}, 2023.

\bibitem{GLL}
Olivier Gu\'{e}ant, Jean-Michel Lasry, and Pierre-Louis Lions.
\newblock Mean field games and applications.
\newblock In {\em Paris-{P}rinceton {L}ectures on {M}athematical {F}inance
  2010}, volume 2003 of {\em Lecture Notes in Math.}, pages 205--266. Springer,
  Berlin, 2011.

\bibitem{guo2023ito}
Xin Guo, Huy{\^e}n Pham, and Xiaoli Wei.
\newblock It{\^o}’s formula for flows of measures on semimartingales.
\newblock {\em Stochastic Processes and their Applications}, 159:350--390,
  2023.

\bibitem{hafayed2014mean}
Mokhtar Hafayed, Abdelmadjid Abba, and Syed Abbas.
\newblock On mean-field stochastic maximum principle for near-optimal controls
  for {P}oisson jump diffusion with applications.
\newblock {\em International Journal of Dynamics and Control}, 2:262--284,
  2014.

\bibitem{hafayed2018optimal}
Mokhtar Hafayed, Shahlar Meherrem, {\c{S}}aban Eren, and Deniz~Hasan
  Gu{\c{c}}oglu.
\newblock On optimal singular control problem for general {M}ckean-{V}lasov
  differential equations: Necessary and sufficient optimality conditions.
\newblock {\em Optimal Control Applications and Methods}, 39(3):1202--1219,
  2018.

\bibitem{hammersley2021weak}
William~RP Hammersley, David {\v{S}}i{\v{s}}ka, and {L}ukasz Szpruch.
\newblock Weak existence and uniqueness for {M}ckean--{V}lasov {SDE}s with
  common noise.
\newblock {\em The Annals of Probability}, 49(2):527--555, 2021.

\bibitem{hu2017singular}
Yaozhong Hu, Bernt {\O}ksendal, and Agnes Sulem.
\newblock Singular mean-field control games.
\newblock {\em Stochastic Analysis and Applications}, 35(5):823--851, 2017.

\bibitem{ito1951formula}
Kiyosi It{\^o}.
\newblock On a formula concerning stochastic differentials.
\newblock {\em Nagoya Mathematical Journal}, 3:55--65, 1951.

\bibitem{jacod1981weak}
Jean Jacod and Jean M{\'e}min.
\newblock Weak and strong solutions of stochastic differential equations:
  existence and stability.
\newblock In {\em Stochastic integrals}, pages 169--212. Springer, 1981.

\bibitem{jeanblanc2010immersion}
Monique Jeanblanc and Yann Le~Cam.
\newblock Immersion property and credit risk modelling.
\newblock {\em Optimality and Risk-Modern Trends in Mathematical Finance: The
  Kabanov Festschrift}, pages 99--132, 2010.

\bibitem{karatzas-shreve}
Ioannis Karatzas and Steven~E. Shreve.
\newblock {\em Brownian motion and stochastic calculus}, volume 113 of {\em
  Graduate Texts in Mathematics}.
\newblock Springer-Verlag, New York, second edition, 1991.

\bibitem{KT}
Vassili~N. Kolokoltsov and Marianna Troeva.
\newblock On mean field games with common noise and {M}ckean-{V}lasov spdes.
\newblock {\em Stochastic Analysis and Applications}, 37(4):522--549, 2019.

\bibitem{Ko}
Peter Kotelenez.
\newblock A class of quasilinear stochastic partial differential equations of
  {M}c{K}ean-{V}lasov type with mass conservation.
\newblock {\em Probab. Theory Related Fields}, 102(2):159--188, 1995.

\bibitem{kurtz2007yamada}
Thomas Kurtz.
\newblock The {Y}amada--{W}atanabe--{E}ngelbert theorem for general stochastic
  equations and inequalities.
\newblock 2007.

\bibitem{kurtz2014weak}
Thomas Kurtz.
\newblock Weak and strong solutions of general stochastic models.
\newblock {\em Electronic Communications in Probability}, 19, 2014.

\bibitem{kurtz1999particle}
Thomas~G Kurtz and Jie Xiong.
\newblock Particle representations for a class of nonlinear {SPDE}s.
\newblock {\em Stochastic Processes and their Applications}, 83(1):103--126,
  1999.

\bibitem{La}
Daniel Lacker.
\newblock A general characterization of the mean field limit for stochastic
  differential games.
\newblock {\em Probab. Theory Related Fields}, 165(3-4):581--648, 2016.

\bibitem{lacker2016general}
Daniel Lacker.
\newblock A general characterization of the mean field limit for stochastic
  differential games.
\newblock {\em Probability Theory and Related Fields}, 165:581--648, 2016.

\bibitem{lacker2022superposition}
Daniel Lacker, Mykhaylo Shkolnikov, and Jiacheng Zhang.
\newblock Superposition and mimicking theorems for conditional
  {M}ckean--{V}lasov equations.
\newblock {\em Journal of the European Mathematical Society}, 2022.

\bibitem{LW}
Daniel Lacker and Kevin Webster.
\newblock Translation invariant mean field games with common noise.
\newblock {\em Electron. Commun. Probab.}, 20:no. 42, 13, 2015.

\bibitem{li2018mean}
Juan Li.
\newblock Mean-field forward and backward {SDE}s with jumps and associated
  nonlocal quasi-linear integral-{PDE}s.
\newblock {\em Stochastic Processes and their Applications}, 128(9):3118--3180,
  2018.

\bibitem{nicole1987compactification}
Karoui Nicole~el, Nguyen Du'h{\=U}{\=U}, and Jeanblanc-Picqu{\'e} Monique.
\newblock Compactification methods in the control of degenerate diffusions:
  existence of an optimal control.
\newblock {\em Stochastics: {A}n {I}nternational {J}ournal of {P}robability and
  {S}tochastic {P}rocesses}, 20(3):169--219, 1987.

\bibitem{protter2005stochastic}
Philip~E Protter.
\newblock {\em Stochastic {D}ifferential {E}quations}.
\newblock Springer, 2005.

\bibitem{reis2021relation}
Goncalo~dos Reis and Vadim Platonov.
\newblock On the relation between {S}tratonovich and {I}t{\^o} integrals with
  functional integrands of conditional measure flows.
\newblock {\em arXiv preprint arXiv:2111.03523}, 2021.

\bibitem{talbi2023dynamic}
Mehdi Talbi, Nizar Touzi, and Jianfeng Zhang.
\newblock Dynamic programming equation for the mean field optimal stopping
  problem.
\newblock {\em SIAM Journal on Control and Optimization}, 61(4):2140--2164,
  2023.

\bibitem{Va}
Jean Vaillancourt.
\newblock On the existence of random {M}c{K}ean-{V}lasov limits for triangular
  arrays of exchangeable diffusions.
\newblock {\em Stochastic Anal. Appl.}, 6(4):431--446, 1988.

\end{thebibliography}

\end{document}